\documentclass[12pt]{amsart}

\usepackage{tikz}
\usetikzlibrary{snakes}
\usepackage[colorlinks=true, linkcolor=blue, anchorcolor=blue, citecolor=blue, filecolor=blue, menucolor= blue, urlcolor=blue]{hyperref}
\usepackage{amsmath,amssymb,amsthm,amscd}
\usepackage{verbatim}
\usepackage{comment}
\usepackage{multirow}
\usepackage{mathtools}
\usepackage{enumitem}
\usepackage{pgf,tikz}
\usetikzlibrary{arrows}
\usepackage[normalem]{ulem}

\usepackage[margin=1in]{geometry} 

\numberwithin{equation}{section}

\newtheorem{theorem}{Theorem}[section]
\newtheorem{proposition}[theorem]{Proposition}
\newtheorem{lemma}[theorem]{Lemma}
\newtheorem{corollary}[theorem]{Corollary}
\newtheorem{problem}[theorem]{Problem}
\newtheorem{question}[theorem]{Question}

\newtheorem*{theorem*}{Theorem}

\theoremstyle{definition}
\newtheorem{definition}[theorem]{Definition}

\newtheorem{example}[theorem]{Example}
\newtheorem{remark}[theorem]{Remark}

\newcommand{\Ab}{ \ensuremath{\mathbb{A}}}
\newcommand{\FF}{ \ensuremath{\mathbb{F}}}
\newcommand{\NN}{ \ensuremath{\mathbb{N}}}
\newcommand{\ZZ}{ \ensuremath{\mathbb{Z}}}
\newcommand{\CC}{ \ensuremath{\mathbb{C}}}
\newcommand{\QQ}{ \ensuremath{\mathbb{Q}}}
\newcommand{\RR}{ \ensuremath{\mathbb{R}}}
\newcommand{\PP}{ \ensuremath{\mathbb{P}}}
\newcommand{\KK}{ \ensuremath{\mathbb{K}}}

\newcommand{\GL}{{GL}_r (K)}
\newcommand{\Gin}{\ensuremath{\mathrm{Gin}}}
\newcommand{\gin}{\ensuremath{\mathrm{gin}}}
\newcommand{\GIN}{\ensuremath{\mathrm{GIN}}}
\newcommand{\Gins}{\ensuremath{\mathrm{Gins}}}
\newcommand{\gins}{\ensuremath{\mathrm{gins}}}
\newcommand{\Lex}{\ensuremath{\mathrm{Lex}}}
\newcommand{\dele}[1]{\ensuremath{\Delta^e({#1})}}
\newcommand{\dels}[1]{\ensuremath{\Delta^s({#1})}}
\newcommand{\delc}[1]{\ensuremath{\Delta^c({#1})}}
\newcommand{\Shift}{\ensuremath{\mathrm{Shift}}}

\newcommand{\cE}{\mathcal{E}}
\newcommand{\mA}{\mathcal{A}}
\newcommand{\mB}{\mathcal{B}}
\newcommand{\mC}{\mathcal{C}}
\newcommand{\mD}{\mathcal{D}}
\newcommand{\mH}{\mathcal{H}}
\newcommand{\mI}{\mathcal{I}}
\newcommand{\mJ}{\mathcal{J}}
\newcommand{\mO}{\mathcal{O}}
\newcommand{\mP}{\mathcal{P}}
\newcommand{\mS}{\mathcal{S}}
\newcommand{\mV}{\mathcal{V}}
\DeclareMathOperator{\Jac}{Jac}
\DeclareMathOperator{\coker}{coker}
\newcommand{\HH}{H_{\mathfrak m}}

\newcommand{\init}{\ensuremath{\mathrm{in}}\hspace{1pt}}
\newcommand{\lex}{{\mathrm{lex}}}
\newcommand{\rlex}{{\mathrm{{rlex}}}}
\newcommand{\rev}{{\mathrm{{rev}}}}

\newcommand{\rank}{\ensuremath{\mathrm{rank}}\hspace{1pt}}
\newcommand{\reg}{\ensuremath{\mathrm{reg}}\hspace{1pt}}
\newcommand{\depth}{\ensuremath{\mathrm{depth}}\hspace{1pt}}
\newcommand{\Tor}{\ensuremath{\mathrm{Tor}}\hspace{1pt}}
\newcommand{\Hom}{\ensuremath{\mathrm{Hom}}\hspace{1pt}}
\newcommand{\Ext}{\ensuremath{\mathrm{Ext}}\hspace{1pt}}
\newcommand{\Ker}{\ensuremath{\mathrm{Ker}}\hspace{1pt}}
\newcommand{\Image}{\ensuremath{\mathrm{Im}}\hspace{1pt}}

\newcommand{\aaa}{\mathbf{a}}
\newcommand{\ee}{\mathbf{e}}
\newcommand{\pdim}{ \ensuremath{\mathrm{proj\ dim}}\hspace{1pt}}
\newcommand{\Hilb}{\mathrm{Hilb}}
\newcommand{\lk}{{\mathrm{lk}}}
\newcommand{\st}{\mathrm{st}}

\renewcommand{\labelenumi}{\arabic{enumi}.  }
\renewcommand{\labelenumii}{(\alph{enumii})}

\newcommand{\mideal}{\ensuremath{\mathfrak{m}}}

\newcommand\meta[1]{{\ \newline\llap{\Huge$\rightarrow$\hskip.75ex}\small\upshape\bfseries$\langle\!\langle\,$#1$\,\rangle\!\rangle$\ \newline}}

\definecolor{MyDarkGreen}{cmyk}{0.7,0,1,0}

\def\cocoa{{\hbox{\rm C\kern-.13em o\kern-.07em C\kern-.13em o\kern-.15em A}}}

\newcommand{\MM}{\mathcal{M}}

\begin{document}

\title[Unexpected hypersurfaces and where to find them]{Unexpected hypersurfaces and where to find them}

\author{B.\ Harbourne}
\address{Department of Mathematics\\
University of Nebraska\\
Lincoln, NE 68588-0130 USA}
\email{bharbourne1@unl.edu}

\author{J.\ Migliore} 
\address{Department of Mathematics \\
University of Notre Dame \\
Notre Dame, IN 46556 USA}
 \email{migliore.1@nd.edu}

\author{U.\ Nagel}
\address{Department of Mathematics\\
University of Kentucky\\
715 Patterson Office Tower\\
Lexington, KY 40506-0027 USA}
\email{uwe.nagel@uky.edu}

\author{Z.\ Teitler}
\address{Department of Mathematics\\
Boise State University\\
1910 University Drive\\
Boise, ID 83725-1555, USA}
\email{zteitler@boisestate.edu}

\begin{abstract} 
In the paper \cite{CHMN} by Cook, et al., which introduced the concept of unexpected plane curves,
the focus was on understanding the geometry of the curves themselves.
Here we expand the definition to hypersurfaces of any dimension
and, using constructions which appeal to algebra, geometry, representation 
theory and computation, we obtain a coarse but complete classification of 
unexpected hypersurfaces. In particular, 
we determine each $(n,d,m)$ for which there is some finite set of points $Z\subset\PP^n$
with an unexpected hypersurface of degree $d$ in $\PP^n$ having a 
general point $P$ of multiplicity $m$. 
Our constructions also give new insight into the interesting question of where to look for such $Z$. Recent 
work of Di Marca, Malara and Oneto \cite{DMO} and of Bauer, Malara, Szemberg and Szpond \cite{BMSS} give new results and 
examples in $\PP^2$ and  $\PP^3$. We obtain our main results using a new construction of unexpected hypersurfaces
involving cones.  
This method applies in $\PP^n$ for $n \geq 3$ and gives a broad range of examples, which we link to certain failures of the 
Weak Lefschetz Property. We also give constructions using root systems, both in $\PP^2$ and $\PP^n$ 
for $n \geq 3$. Finally, we explain an observation of \cite{BMSS}, showing that the unexpected curves of \cite{CHMN} are in 
some sense dual to their tangent cones at their singular point.
\end{abstract}

\date{December 19, 2018}

\thanks{
{\bf Acknowledgements}: Harbourne was partially supported by Simons Foundation grant \#524858.
Migliore was partially supported by Simons Foundation grant \#309556.
Nagel was partially supported by Simons Foundation grant \#317096.
Teitler was partially supported by Simons Foundation grant \#354574.
We thank M.\ Dyer for suggesting that we look at root systems and for 
bringing to our attention the $H_3$ and $H_4$ root systems. We thank Tomasz Szemberg and Justyna Szpond for 
helpful comments. 
And we thank Boise State University, where some of the work
on this paper was done.}

\keywords{fat points, line arrangements, hyperplane arrangements, linear systems, root systems, reflection groups, supersolvable,
stable vector bundle, splitting type}

\subjclass[2010]{14N20 (primary); 13D02, 14C20, 14N05, 05E40,  14F05 (secondary)}

\maketitle



\section{Introduction}

The paper \cite{CHMN} defined the concept of unexpected plane curves and characterized their geometry.
Given a finite set of points $Z\subset\PP^n$ with $n=2$, we say that plane curves of degree $d=m+1$ 
containing $Z$ and having a general point of multiplicity $m$ are {\it unexpected} when such curves exist
but the conditions imposed by vanishing to order $m$ at $P$ on the vector space of forms of degree $d$ 
vanishing on $Z$ are not independent.

Here we broaden the definition in a natural way by allowing $n\geq2$ and $d\geq m$ and we ask
for a coarse classification of unexpected {\it hypersurfaces}. I.e., we ask for which
$(n,d,m)$ there is a finite set of points $Z\subset \PP^n$ for which there is an unexpected hypersurface
of degree $d$ containing $Z$ with a general point of multiplicity $m$. We give a complete
answer to this question in Theorem \ref{mainThm}.

To describe the problem and our approach to it in a more precise way, let $\KK$ be an algebraically closed field of characteristic 0. 
Let $R=\KK[\PP^n]=\KK[x_0,\ldots,x_n]$ be the homogeneous coordinate ring of 
$n$-dimensional projective space.
Consider a general point $P\in\PP^n$.
The fat point scheme $X=mP$ is the scheme defined by the homogeneous ideal
$I_X=(I_{P})^{m} \subset R$, where $I_{P}$ is the ideal generated by all forms that vanish at $P$.
Given a homogeneous ideal $I\subseteq R$, we denote by $[I]_d$ the $\KK$-vector space spanned by homogeneous forms in 
$I$ of degree $d$. It is well known and easy to show that $\dim_\KK[I_X]_d=\max(0,\binom{n+d}{n}-\binom{m-1+n}{n})$.
Given distinct points $Q_i\in\PP^n$, we define $Z=Q_1+\cdots +Q_s$ to be the subscheme defined
by the ideal $I_Z=\cap_jI_{Q_j}\subset R$. 

We say that $Z\subset\PP^n$ {\em admits an unexpected hypersurface} with respect to $X$ of degree $d$ if
$$\dim [I_Z\cap I_X]_d > \max \left \{ 0, \dim [I_Z]_d-\binom{m-1+n}{n} \right \}.$$
That is, $Z$ admits an unexpected hypersurface with respect to $X$ of degree $d$ if the conditions imposed by $X$
on forms of degree $d$ vanishing on $Z$ are not independent. We will also sometimes say that $Z$ admits an unexpected 
hypersurface with a general point $P$ of multiplicity~$m$.

In \cite{CHMN} a version of the following problem was posed for $n=2$:

\begin{problem}\label{prob:intro}
Characterize and then classify all  quadruples $(n,d,m,Z)$
where $Z\subset\PP^n$ admits an unexpected 
hypersurface of degree $d$ with a general point $P$ of multiplicity $m$.
\end{problem}

As a means for approaching this problem, and motivated by an example in \cite{DIV}, the paper \cite{CHMN} gave a 
careful analysis for $n=2$ of the structure of unexpected hypersurfaces (hence curves) when $d=m+1$.
Results of \cite{DIV} and \cite{FV} show that it is useful in this situation to regard a point set $Z$ as the 
points $Z_\mA$ dual to a line arrangement $\mA$.

Given an arrangement $\mA$ of lines in $\PP^2$, the results of \cite{CHMN} provide a means for determining whether
the reduced scheme $Z_\mA$ of points dual to the lines admits an unexpected curve using invariants of $\mA$, but it is still very unclear
which line arrangements $\mA$ to look at.
One of the best results in this special case is that of \cite{DMO}, completely characterizing the supersolvable (see definition in \S\ref{root system section}) line arrangements $\mA$
such that the points $Z_\mA$ admit an unexpected curve. 
Both supersolvable and non-supersolvable line arrangements were  studied in \cite{CHMN}, and the latter can 
also give rise to unexpected curves, but it is not clear which ones do. 

Moreover, the only examples heretofore in the literature
of unexpected hypersurfaces with an imposed singularity at a single general point in $\PP^n$ for $n>2$ 
that we are aware of are examples given in \cite{BMSS}.
But given that they do occur, it is a natural and interesting next step to work to understand the range of examples of unexpected hypersurfaces that can occur
with an imposed singularity at a single general point in $\PP^n$,
both in dimension 2 and in higher dimensions,
and to find structural connections between the geometry of a reduced finite set of points $Z$ in $\PP^n$ for $n\geq 2$
and the existence of such an unexpected hypersurface. 

In \S\ref{cone section} we show that a large class of examples is related to $Z$ lying on projective cones 
over codimension 2 subvarieties. It ends with an application to the question of when ideals generated by 
powers of linear forms fail the Weak Lefschetz Property.
In \S\ref{root system section} we show that root systems can give rise to real point sets $Z$
admitting unexpected hypersurfaces. There were already indications
 in \cite{CHMN} (see also \cite{I}) that hyperplane arrangements related to reflection groups sometimes 
give rise to unexpected hypersurfaces (see what was called the Fermat, Klein and Wiman arrangements in \cite{CHMN};
these all come from complex reflection groups). We reinforce these indications here by finding additional
examples coming from root systems of real reflection groups. These have the advantage
of providing obvious candidates in higher dimension too.
However, not all root systems seem to give rise to unexpected hypersurfaces;
it would be an interesting project to understand what is special about those that do. 
In \S\ref{BMSSduality} we present initial results regarding a still mysterious duality between 
unexpected hypersurfaces having an imposed general singular point $P$,
and their tangent cone at $P$, first observed in \cite{BMSS}. Finally, in \S\ref{open problems} we present some open questions arising from our work.

By applying our results from \S\ref{cone section},
we obtain the following theorem, one of the main consequences of our work in this paper.
See Theorem \ref{m<d} for the proof.

\begin{theorem}\label{mainThm}
Given positive integers $(n,d,m)$ with $n>1$, there exists an unexpected hypersurface for some finite set of points 
$Z \subset  \PP^n$ of degree $d$ with a general point of multiplicity $m$ if and only if one of the following conditions holds:
\begin{enumerate}[label={(\roman*)}]
\item[(i)] $n=2$ and $(d,m)$ satisfies $d > m > 2$; or
\item[(ii)] $n \ge 3$ and $(d,m)$ satisfies $d \geq m \geq 2$.
\end{enumerate}
\end{theorem}

It is still very unclear what kinds of unexpected hypersurfaces can occur for each $d$ and $m$.
One goal of this paper was to suggest new venues for where to find them. 
The title of this paper should not, however, be taken to mean that we have found all unexpected hypersurfaces (a title for a possible
future paper could be ``Unexpected hypersurfaces and where {\it else} to find them"). 
In addition, in contrast to what \cite{CHMN} was able to do in $\PP^2$,
there are not yet good tools in higher dimension for rigorously verifying unexpectedness.
In particular, we are able to give rigorous verifications of unexpectedness for the new examples coming from
root systems only in some of the cases where we suspect that they occur.

\subsection*{Notation}

For any subvariety (or subscheme) $V \subseteq \PP^n$ we write $I_V \subseteq R$ for the saturated ideal of $V$
and $\mathcal{I}_V$ for the sheaf on $\PP^n$ corresponding to $I_V$.
For any integer function $h : \ZZ_{\geq 0} \to \ZZ$ the first difference $\Delta h$ is
the backward difference $\Delta h(t) = h(t)-h(t-1)$, where we make the convention $h(-1)=0$ (so $\Delta h(0) = h(0)$).


\section{Cones} \label{cone section}

In this section we give a method for constructing examples of varieties $Z$ (not necessarily points) with unexpected hypersurfaces.
Although by far the more interesting question is the problem of understanding the unexpected hypersurfaces arising from a finite set 
of points, one can also begin by asking whether a reduced, non-degenerate {\em curve} in $\PP^3$ admits unexpected surfaces. We 
obtain the somewhat surprising fact that they always do! Using B\'ezout's theorem we then translate this back to finite sets of points. 
We also extend this idea to $\PP^n$. Finally, we find a connection to the well-studied question of when an ideal generated by powers 
of linear forms has the Weak Lefschetz Property, extending results of \cite{DIV} who first noticed a connection between cones and WLP.

Our method involves cones.
By a cone with vertex $P$ we mean a scheme $X$ such that for every point $Q$ in $X$ the line joining $P$ and $Q$ is in $X$.
In particular, by B\'ezout, every hypersurface of degree $d$ with a point of multiplicity $d$ at a point $P$ is a cone with vertex $P$.

It is not hard to show that a plane curve of degree $d$ in $\PP^3$ does not admit an unexpected hypersurface with a point of 
multiplicity $d$ -- instead, a point of multiplicity $d$ imposes the expected number of conditions on hypersurfaces of degree $d$ 
containing the plane curve (use the fact that a plane curve in $\PP^3$ is a complete intersection,
and the known Hilbert function for complete intersections). For non-degenerate curves the situation is very different, as we now show.

\begin{proposition} \label{cone in P3}
Let $C$ be a reduced, equidimensional, non-degenerate curve of degree $d$ in $\PP^3$ ($C$ may be reducible, disconnected, and/or singular 
but note that $d \geq 2$ since $C$ is non-degenerate, with $C$ being two skew lines if $d=2$). Let $P \in \PP^3$ be a general point. 
Then the cone $S_P=S_P(C)$ over $C$ with vertex $P$ is an unexpected surface of degree $d$ for $C$ with multiplicity $d$ at $P$. It is the unique unexpected surface of this degree and multiplicity.
\end{proposition}

\begin{proof}

We first check uniqueness.  Let $F$ be a form defining a surface containing $C$, of degree $d$ with 
multiplicity $d$ at $P$. Let $\lambda$ be a line through $P$ and any point $Q$, of $C$. Then  by 
B\'ezout, $F$ must vanish on all of $\lambda$. Thus the surface defined by $F$ is precisely $S_P$.

We now check unexpectedness. Let $D$ be a smooth plane curve of degree $d$. 

\medskip

\noindent \underline{Claim 1}: {\em The arithmetic genus, $g_C$, of $C$ is strictly less than that of $D$, which is $g_D = \binom{d-1}{2}$. }

\medskip

This argument is classical. Much of it is given in \cite{harris} (when C is irreducible) and in \cite{migbook} Proposition 1.4.2, so no claim is made to originality; we include it here just for the reader's convenience.
Let $\Gamma$ be a general hyperplane section of $C$ by a hyperplane $H$ defined by a general linear form $L$. Let $I_{\Gamma | H}$ be the saturated ideal of $\Gamma$ in $H$. Let $\ell \gg 0$. Then
\[
\begin{array}{rcl}
d\ell - g_C + 1 & = & h^0(\mathcal O_C(\ell)) \\
& = & h_{R/I_C}(\ell)
\end{array}
\]
where $h_{R/I_C}(t)$ is the Hilbert function of $C$. On the other hand, for any integer $t$ we have the exact sequence
\[
\begin{array}{ccccccccccccccccccc}
0 & \rightarrow & [I_C]_{t-1} & \rightarrow & [I_C]_t & \rightarrow & [I_{\Gamma|H} ]_t & \longrightarrow & H^1(\mathcal I_C(t-1)) & \rightarrow & \dots \\
&&&&&&\hbox{\ \ \ \ \ \ \ } \searrow & & \nearrow \hbox{\ \ \ \ \ \ \ \ \ \ \ \ } \\
&&&&&&& [K]_t \\
&&&&&&\hbox{\ \ \ \ \ \ \ } \nearrow & & \searrow \hbox{\ \ \ \ \ \ \ \ \ \ \ \ } \\
&&&&&&\hbox{\  } 0 & & 0 \hbox{\ \ \ \ \ \  } \\
\end{array}
\]
(where $K$ is just the cokernel). Then after adding and subtracting some binomial coefficients and setting $\bar R = R/L$, we obtain
\[
\begin{array}{rcl}
\Delta h_{R/I_C} (t) & = & h_{\bar R/I_{\Gamma|H}}(t) + \dim [K]_t \\ 
& \geq & h_{\bar R/I_{\Gamma | H}} (t).
\end{array}
\]
So since $\ell \gg 0$ we obtain
\[
\begin{array}{rcl}
g_C & = & 1 + d\ell - h_{R/I_C}(\ell) \\
& = & d\ell - \sum_{t=1}^\ell \Delta h_{R/I_C}(t) \\
& \leq & d\ell - \sum_{t=1}^\ell  h_{\bar R/I_{\Gamma|H}} (t) \\
& = & \sum_{t=1}^\ell \left [ d - h_{\bar R/I_{\Gamma | H}} (t) \right ].
\end{array}
\]
Now replace $C$ by $D$, and replace $\Gamma$ by the hyperplane section of $D$, which is a set of $d$ collinear points, say $A$.
We have, similarly,
\[
  g_D = \sum_{t=1}^\ell \left [ d - h_{\bar R/I_A}(t) \right ] = \binom{d-1}{2} .
\]
It is clear that for any $t \geq 1$ we have 
\[
h_{\bar R/I_{\Gamma|H}} (t) \geq h_{\bar R/I_A}(t)
\]
with strict inequality for $t=1$, 
so we obtain $g_C < g_D$. This completes the proof of Claim 1.

Now, by \cite{GLP} Remark (1) (p. 497), we have $H^1(\mathcal I_C(d)) = H^2(\mathcal I_C(d)) = 0$, and we also have $H^1(\mathcal I_D(d)) = H^2(\mathcal I_D(d)) = 0$.
Consider the exact sequence
\[
0 \rightarrow [I_C]_d \rightarrow H^0(\mathcal O_{\PP^3}(d)) \rightarrow H^0(\mathcal O_C(d)) \rightarrow H^1(\mathcal I_C(d)) \rightarrow 0.
\]

\medskip

\noindent \underline{Claim 2}:
{\em Let $D$ be a plane curve of degree $d$. Then $h^0(\mathcal O_C(d)) > h^0(\mathcal O_D(d))$.}

\medskip

 Indeed,
\[
\begin{array}{rcl}
d^2 - g_C + 1 & = & h^0(\mathcal O_C(d)) - h^1(\mathcal O_C(d)) \\
& = & h^0(\mathcal O_C(d)) - h^2(\mathcal I_C(d)) \\
& = & h^0(\mathcal O_C(d))
\end{array}
\]
and similarly for $D$, so thanks to Claim 1 we have
\[
h^0(\mathcal O_C(d)) > h^0(\mathcal O_D(d))
\]
and we have Claim 2.

Then
\[
\begin{array}{rcll}
 \dim [I_C]_d & = & \displaystyle \binom{d+3}{3} - h^0(\mathcal O_C(d)) & \hbox{(since } h^1(\mathcal I_C(d)) = 0 ) \\ \\
& < & \displaystyle \binom{d+3}{3} - h^0(\mathcal O_D(d))
\end{array}
\]
thanks to Claim 2. To check that $S_P$ is an unexpected surface it is enough to note that
\[
\begin{array}{rcl}
\displaystyle \dim [I_C]_d - \binom{d-1+3}{3} & < & \displaystyle \binom{d+3}{3} - h^0(\mathcal O_D(d)) - \binom{d-1+3}{3} \\ \\
& = & \displaystyle \binom{d+2}{2} - [d^2 - g_D + 1] \\ \\
& = & 1
\end{array}
\]
using the value of $g_D$ mentioned above. 
\end{proof}

\begin{remark} \label{low deg rmk}
The same argument that was used to show uniqueness in the last result also shows that if $C$ is a curve of degree $d$ in $\PP^3$ 
then there does not exist a surface of degree $e \leq d-1$ containing $C$ with a singularity of multiplicty $e$ at a general point, since 
B\'ezout would force $S_P$ to be a component of such a surface. This means that $\dim [I_C]_e \leq \binom{(e-1)+3}{3}$ for all 
$1 \leq e \leq d$ (where the statement for degree $d$ is given in the proof).

In fact, the same argument works for a subvariety $V$ of $\PP^n$ of codimension two and degree $d$, to show that 
$\dim [I_V]_e \leq \binom{(e-1)+n}{n}$ for $1 \leq e \leq d-1$. This statement also extends to the case $e=d$, and the argument is contained in the proof of Proposition \ref{Pn}.

\end{remark}

We will give several  corollaries to Proposition \ref{cone in P3}. The first is to extend it to subvarieties of codimension two in $\PP^n$.

\begin{lemma} \label{nondeg}
Let $V$ be a reduced, equidimensional, non-degenerate subvariety of $\PP^n$ ($n \geq 4$) of codimension 2 (not necessarily irreducible). 
Let $H$ be a general hyperplane and let $W = V \cap H$. Then $W$ is non-degenerate in $H = \PP^{n-1}$.
\end{lemma}

\begin{proof}
We have the exact sequence
\[
0 \rightarrow [I_V]_0 \rightarrow [I_V]_1 \rightarrow [I_{W|H}]_1 \rightarrow H^1(\mathcal I_V (0)) \rightarrow \dots
\]
We want to show that the third vector space in this exact sequence is zero. But $[I_V]_1 = 0$ since $V$ is non-degenerate. On the other 
hand, we claim that $H^1(\mathcal I_V(0)) = 0$. This will complete the proof. But we also have the exact sequence
\[
0 \rightarrow [I_V]_0 \rightarrow [R]_0 \rightarrow H^0(\mathcal O_V(0)) \rightarrow H^1(\mathcal I_V(0)) \rightarrow 0.
\]
The first term is clearly zero. The second has dimension 1. The third has dimension 1 since $V$ is connected (being of codimension 2 and 
equidimensional) and reduced (by hypothesis). Thus the claim follows.
\end{proof}

\begin{proposition} \label{Pn}
Let $V$ be a reduced, equidimensional, non-degenerate subvariety of $\PP^n$ ($n \geq 3$) of codimension 2 and degree $d$ ($V$ may be 
reducible and/or singular but note that $d \geq 2$ since $V$ is non-degenerate, with $V$ being two codimension 2 linear spaces if $d=2$). 
Let $P \in \PP^n$ be a general point.   
Then the cone $S_P$ over $V$ with vertex $P$ is an unexpected hypersurface for $V$ of degree $d$ and multiplicity $d$ at $P$. It is 
the unique unexpected hypersurface of this degree and multiplicity.
\end{proposition}

\begin{proof}
The proof is by induction on $n$. The initial case is Proposition \ref{cone in P3}, so we can assume $n \geq 4$.
Let $H$ be a general hyperplane through $P$ and let $W = V \cap H$. Since $P$ is general, we can assume that $H$ is general as well.
By Lemma \ref{nondeg}, $W$ is non-degenerate in $H = \PP^{n-1}$, and it is also reduced and equidimensional. Let $T_P$ be the cone in 
$H$ over $W$ with vertex $P$. Thus by induction, $T_P$ is the unique hypersurface of degree $d$ containing $W$ with multiplicity $d$ at $P$, and it is unexpected.

Consider the exact sequence
\[
\begin{array}{ccccccccccccccccccc}
0 & \rightarrow & [I_V]_{d-1} & \rightarrow & [I_V]_d & \rightarrow & [I_{W|H} ]_d & \longrightarrow & H^1(\mathcal I_V(d-1)) & \rightarrow & \dots \\
&&&&&&\hbox{\ \ \ \ \ \ \ } \searrow & & \nearrow \hbox{\ \ \ \ \ \ \ \ \ \ \ \ } \\
&&&&&&& [K]_d \\
&&&&&&\hbox{\ \ \ \ \ \ \ } \nearrow & & \searrow \hbox{\ \ \ \ \ \ \ \ \ \ \ \ } \\
&&&&&&\hbox{\  } 0 & & 0 \hbox{\ \ \ \ \ \  } \\
\end{array}
\]
We have
\[
\dim [I_V]_d = \dim [I_V]_{d-1} + \dim [I_{W|H}]_d - \dim [K]_d.
\]
We claim that  
\[
\dim [I_V]_d - \binom{d-1+n}{n} \leq 0
\]
so that $S_P$ is unexpected. We have
\[
\begin{array}{rcl}
\displaystyle \dim [I_V]_d - \binom{d-1+n}{n} & = & \displaystyle \dim [I_V]_{d-1}  + \left ( \dim [I_{W|H}]_d - \binom{d-1+(n-1)}{n-1} \right ) \\ \\
&& \displaystyle + \binom{d-1+(n-1)}{n-1} - \dim [K]_d - \binom{d-1+n}{n} \\ \\
& \leq & \displaystyle  \dim [I_V]_{d-1} + \binom{d+n-2}{n-1} - \dim [K]_d - \binom{d+n-1}{n} \\ \\
& = & \displaystyle \dim [I_V]_{d-1} - \binom{d+n-2}{n} - \dim [K]_d  .
\end{array}
\]
But Remark \ref{low deg rmk} gives
\[
\dim [I_V]_{d-1} \leq \binom{d+n-2}{n},
\]
which completes the claim.  Uniqueness follows in the same way as it did for Proposition~\ref{cone in P3}.
\end{proof}

The next few corollaries have analogs in higher projective space, but the statements are a bit cleaner for curves in $\PP^3$.

\begin{corollary} \label{curves 1}
Let $C$ be a reduced, equidimensional, non-degenerate curve of degree $d$ in $\PP^3$ ($C$ may be reducible, singular, and/or disconnected). 
Let $P \in \PP^3$ be a general point. Let $Z \subset C$ be any set of  points on $C$ such that $[I_C]_d = [I_Z]_d$. 
Then the cone $S_P$ over $C$ with vertex $P$ is an unexpected surface of degree $d$ for $Z$ with multiplicity $d$ at $P$. It is the 
unique unexpected surface of this degree and multiplicity. In particular, we may choose $Z$ to impose independent conditions on forms of degree $d$.
\end{corollary}

\begin{proof}
It is immediate from the hypothesis that $[I_C]_d = [I_Z]_d$.
\end{proof}

\begin{corollary} \label{curves 2}
Let $C$ be a smooth, irreducible, non-degenerate curve of degree $d \geq 3$ in $\PP^3$. Let $P \in \PP^3$ be a general point. 
Let $Z \subset C$ be any set of at least $d^2+1$ points on $C$ (general or not). Then the cone $S_P$ over $C$ with vertex $P$ is 
an unexpected surface of degree $d$ for $Z$ with multiplicity $d$ at $P$. It is the unique unexpected surface of this degree and multiplicity.
\end{corollary}

\begin{example} \label{twc1}
Let $C$ be a twisted cubic curve in $\PP^3$. Then $d=3$ and $g_C = 0$. Let $Z$ be a set of $10$ points on $C$, so $[I_C]_3 = [I_Z]_3$ 
has dimension $10$. In this case $\dim [I_Z]_3 - \binom{2+3}{3} = 10 - 10 = 0$ so we do not expect a hypersurface of degree 3 with 
multiplicity 3 at a general point containing the 10 points of $Z$. But in fact there is such an unexpected hypersurface, given by the cone over $C$ with vertex at a general point.
\end{example}

\begin{remark} \label{need fewer pts}
Corollary \ref{curves 1}, Corollary \ref{curves 2}, Corollary \ref{add S} and Corollary \ref{WLP Pn cons} all deal with the situation 
that we begin with a set of points lying on a variety $C$ of codimension two in $\PP^n$, and have enough points so that 
$[I_C]_d = [I_Z]_d$. In fact this assumption can be relaxed, although the statement becomes a little bit less transparent so we 
retained this assumption. But notice that the fact that $C$ already admits an unexpected hypersurface of degree $d$ means that 
we only need a set of $\binom{d+n}{n} - \binom{d-1+n}{n}$ points on $C$ that impose independent conditions on forms of degree 
$d$, and this number can be much smaller than the number forced by the condition $[I_C]_d = [I_Z]_d$.

For example, say $C$ is a general smooth rational curve in $\PP^3$ of degree 6. The Hilbert function of $C$ is given by the 
sequence $1,4,10, 19, 25, 31, 37, \dots$ so the assumption that $[I_C]_6 = [I_Z]_6$ means we need $Z$ to have at least 37 
points of $C$. Instead, suppose that $Z$ is a sufficiently general set of $\binom{6+3}{3} - \binom{5+3}{3} = 28$ points on $C$. 
Then the Hilbert function of $Z$ is given by the sequence $1,4,10,19, 25, 28, 28, \dots$ and we still do not expect a hypersurface of 
degree 6 with a point of multiplicity 6 to contain $Z$, but we know that the cone over $C$ is such a hypersurface. Notice that in this 
case we not only have $[I_C]_6 \neq [I_Z]_6$ but even $[I_C]_5 \neq [I_Z]_5$.
\end{remark}

\begin{corollary}
Let $C$ be a non-degenerate union of $d$ lines  in $\PP^3$. Let $P \in \PP^3$ be a general point. Let $Z \subset C$ be a  
set of $d(d+1)$ points on $C$ chosen by taking $d+1$ general points on each line. Then the cone $S_P$ over $C$ with vertex 
$P$ is an unexpected surface of degree $d$ for $Z$ with multiplicity $d$ at $P$. It is the unique unexpected surface of this degree and multiplicity.
\end{corollary}

\begin{remark} \label{no unexp for genl pts}
On the other hand, it is not the case that all sets of points in $\PP^3$ (or any other projective space) admit an 
unexpected surface (resp. hypersurface) of some sort. Indeed, suppose $Z$ is a general set of points in $\PP^3$ and let us 
ask if there is any degree and multiplicity at a general point, in which $Z$ admits an unexpected surface. By considering the 
conditions imposed first by the general multiple point and then by the general points $Z$, we see that we must always get the expected number of conditions.

What is interesting is that in \cite{CHMN} Corollary 6.8 it was shown that a set of points in linear general position in $\PP^2$ does not 
admit an unexpected curve of degree $d$ and multiplicity $d-1$ at a general point. Example \ref{twc1} already shows that this does not 
extend to a set of points in linear general position in $\PP^3$, if we weaken the condition on the multiplicity to allow multiplicity $d$. We do 
not know if the precise result from \cite{CHMN} continues to hold in higher dimensional projective spaces.
\end{remark}

\begin{question} \label{GLPQ}
Let $Z$ be a non-degenerate set of points in linear general position in $\PP^n$, $n \geq 3$. Is it true that there does not 
exist an unexpected hypersurface of any degree $d$ and multiplicity $d-1$ at a general point?
\end{question}

We next extend the cone construction  in two different ways. First, we point out that Proposition \ref{cone in P3} extends to surfaces in $\PP^3$ of higher degree and higher multiplicity. 
At the end of this section we will apply this result to show the failure of the Weak Lefschetz Property for certain ideals of powers of linear forms in four variables.

\begin{corollary} \label{higher deg}
Let $C$ be a reduced, equidimensional, non-degenerate curve of degree $d \geq 2$ in $\PP^3$ ($C$ may be reducible, singular, and/or disconnected). 
Let $P \in \PP^3$ be a general point. 
Let $k \geq d$ be a positive integer. Then $C$ admits an unexpected surface of degree $k$ with multiplicity $k$ at $P$.
\end{corollary}

\begin{proof}
Let $Y = C \cup kP$. We want to show that 
\begin{equation} \label{to show}
  \max \left\{ 0 , \dim[I_C]_k - \binom{k+2}{3} \right\} < \dim [I_Y]_k. 
\end{equation}
Modifying the calculation above, we know that
\[
\dim [I_C]_k = \binom{k+3}{3} - [dk - g_C + 1]
\]
so
\[
\dim [I_C]_k - \binom{k+2}{3} = \binom{k+2}{2} - dk + g_C -1.
\]
On the other hand, we have a unique surface $S_P$ of degree $d$ with a singularity of multiplicity $d$ at the general point $P$, so 
by multiplying $S_P$ by an element of $[I_P^{k-d}]_{k-d}$ we always obtain a surface of degree $k$ with  multiplicity $k$ at $P$. Thus
\[
\dim [I_Y]_k \geq \binom{k-d+2}{2},
\]
in particular $\dim [I_Y]_k > 0$. 
Thus combining, it is enough to show
\[
\binom{k+2}{2} - dk + g_C -1 < \binom{k-d+2}{2}.
\]
A calculation shows that this is equivalent to 
\[
g_C < \frac{(d-1)(d-2)}{2},
\]
which we showed in Claim 1 of Proposition \ref{cone in P3}.
\end{proof}

\begin{remark}
Although we do not state them explicitly, we get the analogous corollaries for ``sufficiently many" points on $C$ that we got for 
Proposition \ref{cone in P3}, but now in higher degree. The key is to assume (directly or by a condition on the number of points) that $[I_C]_k = [I_Z]_k$.
\end{remark}

We now give a different extension of the cone construction, allowing us to find unexpected hypersurfaces where the multiplicity is strictly 
less than the degree. It misses by 1 to be an answer to Question \ref{GLPQ}.  

\begin{corollary} \label{add S}
Let $V$ be a reduced, equidimensional, non-degenerate subvariety of codimension two and degree $d$ in $\PP^n$, $n \geq 3$. Let $S$ be a hypersurface of degree $e \geq 1$
not containing any irreducible component of $V$. Let $Y = V \cup S$. Let $Z \subset Y$ be a finite set of 
points such that $[I_Z]_{d+e} = [I_Y]_{d+e}$. Let $P$ be a general point in $\PP^n$. Then $Z$ admits a unique unexpected 
hypersurface of degree $d+e$ with multiplicity $d$ at $P$. In particular, if $V$ is irreducible and $e \geq 2$ then we can take $Z$ 
to be points in linear general position which impose independent conditions on forms of degree $d+e$.
\end{corollary}

\begin{proof}
Let $F$ be the form defining $S$. Then
\[
 [I_Z]_{d+e} =  [I_Y]_{d+e} = F \cdot [I_V]_d
\]
so
\[
\dim [I_Z]_{d+e} - \binom{d+n-1}{n} = \dim [I_V]_d - \binom{d+n-1}{n} < 1
\]
as we saw in the proof of Proposition \ref{Pn}. Thus $S_P \cup S$ is an unexpected hypersurface of degree $d+e$ with a singular point of multiplicity $d$ at $P$.
\end{proof}

With the above results we can now give a complete answer to the following natural question. For which $d$ and $m$ does there exist a 
set of points $Z$ in $\PP^2$ (resp. $\PP^n$) such that $Z$ admits an unexpected curve (resp. hypersurface) of degree $d$ and multiplicity $m$ at a general point?

\begin{theorem} \label{m<d}\ 

\begin{enumerate}[label={(\roman*)}]
\item There exists a finite set of points $Z \subset \PP^2$ admitting an unexpected curve of degree $d$ and multiplicity $m$ at a general point if and only if $d > m > 2$.

\item For $n \geq 3$, there exists a finite set of points $Z \subset \PP^n$ admitting an unexpected hypersurface of degree $d$ and multiplicity $m$ at a general point if and only if $d \geq m \geq 2$.

\end{enumerate}
\end{theorem}

\begin{proof}
For (i), we first note that there is an unexpected curve of degree $m+1$ and multiplicity $m$ at a general point for each $m \geq 3$. Indeed, 

\begin{itemize}

\item $m \leq 2$. It was shown by Akesseh \cite{A} that this occurs only 
in characteristic 2. (See also \cite{FGST} for a different proof that shows that it does not occur in characteristic zero.)

\item $m=3$. This comes from the dual of the $B_3$ arrangement \cite{DIV}. 

\item $m=4$. Consider the line arrangement defined by the linear factors of
$$xyz(x^3-y^3)(x^3-z^3)(y^3-z^3).$$ This is a supersolvable arrangement (see definition in \S\ref{root system section}), so Theorem 3.7 of \cite{DMO} applies. In particular, the 
dual points give a reduced set of 12 points which admit an unexpected curve with $d=5$ and $m=4$.

\item $m=5$. This follows from \cite{CHMN} Proposition 6.15, taking $k=2$. 

\item $m \geq 6$. This comes from the points dual to the Fermat arrangement, by \cite{CHMN} Proposition 6.15, taking $t \geq 5$.

\end{itemize}

\noindent It is  clear that a set of points $Z$ in $\PP^2$ cannot admit an unexpected curve  whose degree and multiplicity at a 
general point $P$ are equal (unlike what we have seen in $\PP^3$). Indeed, in this situation the line joining $P$ to any point of $Z$ 
must be a component of such an unexpected curve, so the unexpected curve is a cone with vertex $P$ over the points of $Z$. 
Let $d = \deg Z$ be the degree of this curve. Then $Z$ imposes independent conditions on curves of degree $d$, so $\dim [I_Z]_d = \binom{d+2}{2} - d$. Then 
\[
\dim [I_Z]_d - \binom{(d-1)+2}{2} = \binom{d+2}{2} - d - \binom{d+1}{2} = (d+1)-d = 1
\]
so in fact the cone is not unexpected. Thus  $m < d$.

With this preparation, we use the same argument as was used to prove Corollary \ref{add S}. Let $Z_0$ be a set of points admitting an 
unexpected curve of degree $m+1$ and multiplicity $m$ at a general point $P$ and let $A$ be a plane curve of degree $d-m-1$ not 
passing through any point of $Z_0$. Since $P$ is general, it does not lie on $A$. Then choosing sufficiently many points on $A$ gives 
a unique unexpected curve of degree $d$ and multiplicity $m$ at $P$. (A similar construction also was used in \cite[Proposition 6.7]{DIV},
but see \cite[Example 7.4]{CHMN}.)

Part (ii) follows from Proposition \ref{Pn} and Corollary \ref{add S}.
\end{proof}

We end this section with an application.  There is an interesting interpretation of these results in terms of the Strong and Weak Lefschetz Properties. 
We first recall the definitions. For these definitions we maintain the assumptions on the polynomial ring $R$, but in fact all we need is that $\KK$ be an infinite field.

\begin{definition}
Let $R/I$ be an artinian $\KK$-algebra and let $L$ be a general linear form. Then $R/I$ satisfies the {\em Weak Lefschetz Property (WLP)} in 
degree $i$ if $\times L : [R/I]_i \rightarrow [R/I]_{i+1}$ has maximal rank, and we say that $R/I$ satisfies the {\em Strong Lefschetz Property 
(SLP)} in degree $i$ with range $k$ if $\times L^k : [R/I]_i \rightarrow [R/I]_{i+k}$ has maximal rank. We say that $R/I$ satisfies WLP (resp.\ SLP) if it does so for all $i$ (resp. for all $i$ and $k$). 
\end{definition}

Thus SLP failing in degree $i$ with range $k$ means that $\times L^k : [R/I]_i \rightarrow [R/I]_{i+k}$ does not 
have maximal rank, and WLP failing in degree $i$ is the same as SLP failing in degree $i$ with range 1.
The result \cite[Theorem 5.1]{DIV} and the remarks that follow the theorem connect the failure of SLP in degree $i$ with range $k$ to the occurrence of a form
on $\PP^n$ of degree $d=i+k$ with a general point of multiplicity $i+1$. In the specific case of $n=k=2$,
\cite[Theorem 7.5]{CHMN} and \cite[Proposition 16]{DI} establish connections between the occurrence
of failures of SLP and existence of unexpected curves (or, in the case of \cite{DI}, something equivalent to the existence of an unexpected curve).
Generalizing the result for unexpected curves to unexpected hypersurfaces, we get the following result, which can be shown to be equivalent to a 
special case  of \cite[Theorem 13]{DI} (which in turn generalizes \cite[Theorem 5.1]{DIV}).
For the reader's convenience, we include a direct proof.

\begin{proposition} \label{SLPvUH}
Let $L_1,\ldots, L_r$ be distinct linear forms on $\PP^n$, and let $Z$ be the set
of points in $\PP^n$ dual to the hyperplanes defined by the $L_i$. Fix integers $d\geq m>1$.
Then the following are equivalent:
\begin{enumerate}
\item[(a)] $Z$ has an unexpected hypersurface of degree $d$ with a general point $P$ of multiplicity~$m$;
\item[(b)] $R/(L_1^{d},\ldots, L_r^{d})$ fails SLP in degree $i=m-1$ with range $k=d-m+1$.
\end{enumerate}
\end{proposition}

\begin{proof}
Let $L$ be a general linear form. Consider the exact sequence
\begin{equation} \label{LES}
\dots \rightarrow [R/(L_1^d, \dots, L_{r}^d)]_{m-1} \stackrel{\times L^{d-m+1}}{\longrightarrow} [R/(L_1^{d},\dots,L_{r}^{d})]_{d} \rightarrow
[R/(L_1^{d},\dots,L_{r}^{d},L^{d-m+1})]_{d} \rightarrow 0.
\end{equation}
Let $P$ be the general point of $\PP^n$ dual to $L$. By Macaulay duality \cite{EI} and exactness, 
\[
\begin{array}{rcl}
\dim [I_Z \cap I_P^{m}]_{d} & = &  \dim [R/(L_1^{d} , \dots, L_r^{d},L^{d-m+1})]_{d} \\ \\
& \geq & \dim [R/(L_1^d,\dots, L_r^d)]_d - \dim [R/(L_1^d,\dots,L_r^d)]_{m-1} \\ \\
& = & \displaystyle \dim [I_Z]_d - \binom{n+m-1}{n}.
\end{array}
\]
Since (a) holds if and only if $\dim [I_Z \cap I_P^{m}]_{d}>\max(0, \dim [I_Z]_d - \binom{n+m-1}{n})$
if and only if $\dim [R/(L_1^{d} , \dots, L_r^{d},L^{d-m+1})]_{d}>\max(0, \dim [R/(L_1^d,\dots, L_r^d)]_d - \dim [R/(L_1^d,\dots,L_r^d)]_{m-1})$
if and only if (b) holds, the result follows.
\end{proof}

Note that a set of points in $\PP^n$ dual to a set of general linear forms is, in particular, in linear general position. 
For a set of points in $\PP^2$ in linear general position (i.e., no three on a line), \cite{CHMN} shows that there is no 
unexpected curve of any degree, hence the corresponding ideals of powers of linear forms do not fail SLP in range 2.
This is in contrast to the case for $n>2$.
Indeed, we now give a result showing failure of WLP for arbitrarily many linear forms in four variables whose dual points are in linear general position 
(but not general), followed in Corollary \ref{WLP Pn cons} by a similar but weaker result for forms in any number of variables $\geq 4$.
(For $\PP^n$ with $n>2$, a few papers have studied the question of WLP for ideals generated by powers of {\em general} linear forms, 
but most such results have focused on a small number of linear forms (e.g. \cite{HSS}, \cite{MMN}, \cite{SS}).)

\begin{corollary} \label{WLP for curves in P3}

Let $C$ be a reduced, irreducible, non-degenerate curve of degree $d \geq 3$ in $\PP^3$. Let $k \geq d$ be a positive integer. 
Let $Z$ be any set of $m \geq dk+1$ points of $C$. Let $L_1,\dots,L_m$ be the linear forms dual to the points of $Z$. In particular, $L_1,\dots,L_m$ 
can be chosen so that no four vanish on a point (i.e. the points of $Z$ are in linear general position). Then $R/(L_1^k,\dots,L_m^k)$ fails the WLP in degree $k-1$.

\end{corollary}

\begin{proof}
The statement about linear general position is immediate since $C$ is reduced, irreducible and non-degenerate. 
From Corollary \ref{higher deg} we know that $C$ has an unexpected surface of degree $k$ with a general point $P$ of multiplicity $k$, so 
we have $\dim [I_C \cap I_P^k]_k>\max(0,\dim [I_C]_k - \binom{k+2}{3})$.
But by B\'ezout's theorem we have $[I_Z]_k = [I_C]_k$ so we obtain $\dim [I_Z \cap I_P^k]_k>\max(0,\dim [I_Z]_k - \binom{k+2}{3})$,
hence $Z$ has an unexpected surface of degree $k$ with multiplicity $k$ at $P$, so 
$R/(L_1^k,\dots,L_m^k)$ fails the WLP in degree $k-1$ by Proposition \ref{SLPvUH}.
\end{proof}

\begin{example}
Let $Z$ consist of a set of 31 points on a twisted cubic $C$. Let $L_1,\dots,L_{31}$ be the linear forms dual to these points. Then for each $3 \leq k \leq 10$,  the algebra
\[
R/(L_1^k,\dots,L_{31}^k)
\]
fails the WLP from degree $k-1$ to degree $k$.

We remark that this does not mean that these are the only powers that fail WLP or that the given degrees are the only places where it fails. 
Experimentally  with CoCoA \cite{cocoa} we have considered the case of 31 points on the twisted cubic as above, but allowed different $k$. 
When $k=2$ the algebra has WLP. When $3 \leq k \leq 10$ it fails from degree $k-1$ to degree $k$ as claimed, but also in certain other 
degrees (depending on $k$). For $k \geq 11$ it still fails in several degrees, but now it does {\em not} fail from degree $k-1$ to $k$.

\end{example}

The following is a slightly weaker  analog of Corollary \ref{WLP for curves in P3} for $\PP^n$.

\begin{corollary} \label{WLP Pn cons}
Let $n \geq 4$ and $k \geq 3$. Set 
\[
f(n, k) = \begin{cases}
\binom{n+k}{n} - \binom{n+k-2}{n} - \binom{n + \frac{k}{2}}{n} + \binom{n -2 + \frac{k}{2}}{n} & \text{ if $k$ is even} \\
\binom{n+k}{n} - \binom{n+k-2}{n} - 2 \binom{n + \frac{k-1}{2}}{n} +  2 \binom{n + \frac{k-3}{2}}{n-1} & \text{ if $k$ is odd.}
\end{cases}
\]
Choose any integer $N \geq f(n,k)$. Then there exist linear forms $L_1,\dots,L_N \in \KK[x_0,\dots,x_n] = R$ satisfying

\begin{itemize}

\item no $n+1$ of the linear forms have a common zero, and

\item $R/(L_1^k,\dots,L_N^k)$ fails the WLP from degree $k-1$ to degree $k$.

\end{itemize}
\end{corollary}

\begin{proof}
We recall that we can always find an irreducible, non-degenerate subvariety $V$  of codimension 2 and degree $k$ in $\PP^n$. 
In fact, if $k$ is even we take $V$ as an intersection of a general form of degree $\frac{k}{2}$ and a general quadric. If $k$ is odd we 
take $V$ as a general arithmetically Cohen-Macaulay subscheme with a minimal free resolution of the form 
\[
0 \rightarrow R(- {\textstyle \frac{k+3}{2} })^2 \rightarrow R(- {\textstyle \frac{k+1}{2} })^2 \oplus R(-2) \rightarrow R \rightarrow R/I_V \rightarrow 0
\]
(obtained, for example, by linking from a linear space of codimension 2 using the complete intersection of a quadric and a form of degree $\frac{k+1}{2}$). 
Using this sequence, or the Koszul resolution if $k$ is even, we obtain
\[
\dim [R/I_V]_k = f (n, k).
\]
Since $V$ is irreducible, it makes sense to speak of a general set of points on $V$. Let $Z$ be a general set of $N$ points on $V$. From the generality of $Z$ we have
\[
\dim [R/I_Z]_k = \min \{ \dim [R/I_V]_k , |Z| \},
\]
hence $\dim [I_V]_k = \dim [I_Z]_k$. By Proposition 2.4 we then have that the cone over $V$ with vertex at a general point $P$ is an unexpected hypersurface of degree $k$ and multiplicity $k$ at $P$.
Then the argument in the proof of Corollary 2.16 can be extended to our situation to show that the multiplication from degree $k-1$ to 
degree $k$ by a general linear form has an unexpectedly large cokernel, i.e. maximal rank does not hold. Since $Z$ is general on $V$ 
and $V$ is irreducible and non-degenerate, $Z$ is a set of points in linear general position, and this implies the  condition on the linear forms not to have a common zero.
\end{proof}

The above results all focus on failure of the WLP. We can also give a result about failure of the SLP with ranges
bigger than 1.

\begin{corollary} \label{fail SLP}
Let $R = \KK[x_0,\dots,x_n]$. Fix positive integers $d\geq m$. Then there exists an ideal $I = (L_1^d,\dots, L_e^d)$ (for suitable $e$) for which
\[
\times L^{d-m+1} : [R/I]_{m-1} \rightarrow [R/I]_d
\]
fails to have maximal rank if and only if one of the following holds.
\begin{itemize}

\item[(i)] We have $n = 2$ and $d > m > 2$.

\item[(ii)] We have $n \geq 3$ and $d \geq m \geq 2$.

\end{itemize}
In both cases this means that $R/I$ fails the SLP in degree $m-1$ and range $d-m+1$.
\end{corollary}

\begin{proof}
This follows immediately from Theorem \ref{m<d} by applying Proposition \ref{SLPvUH}.
\end{proof}

\begin{remark}
\begin{enumerate}
\item Note that Corollary \ref{fail SLP} (and similarly for the earlier results) is not necessarily about failure of surjectivity. The proof only 
shows that the cokernel of $\times L^{d-m+1}$ is bigger than expected, which means failure of surjectivity if $\dim [R/I]_{m-1} \geq \dim [R/I]_d$, and it means failure of injectivity if $\dim [R/I]_{m-1} \leq \dim [R/I]_d$.

\item In \cite{SS}, Schenck and Seceleanu proved the striking result that for {\em every} ideal of the form $I = (L_1^{a_1},\dots,L_e^{a_e})$ in the ring $\KK[x,y,z]$, the quotient $R/I$ has the WLP. In the case that all $a_i$ are equal to some $d$, we can recover this result from Corollary \ref{fail SLP}. Indeed, this is the case $n=2$. Checking whether WLP holds amounts to considering the case $m=d$ (since we are studying $\times L^1$), and Corollary \ref{fail SLP} says that in this case maximal rank always holds.

\end{enumerate}
\end{remark}


\section{Root system examples} \label{root system section}

The construction used in \S \ref{cone section}  
to get an unexpected hypersurface of degree $d$ with a general point of multiplicity $m$
for a locus $Z$ in $\PP^n$ for $n>2$, is based on $[I_Z]_d$ having a positive dimensional base locus.
There are, as we shall see, examples of finite point sets $Z$ with an unexpected hypersurface 
with $d=m$ and $n>2$ such that the base locus of $[I_Z]_d$ is 0 dimensional. Thus our construction
in \S \ref{cone section} is not the end of the story, since unexpected hypersurfaces
can arise in other ways. The question is where else can one look to find them?

In this section we find new habitats where unexpected hypersurfaces lurk, 
both for $n=2$ and $n>2$, at least some of which for $n>2$
have the property that the base locus of $[I_Z]_d$ is 0-dimensional.

We first became aware of a set of points $Z$ admitting an unexpected curve from an example of \cite{DIV}, where
the set $Z$ consists of the points in projective space corresponding to the roots of the $B_3$ root system. 
The lines dual to $Z$ are shown in Figure \ref{B3Fig}.
This example is interesting for a number of reasons.
The line arrangement comes from a root system. (The 
lines in $\PP^2_\RR$ correspond to the 2 dimensional vector subspaces in $\RR^3$ orthogonal to the roots of the $B_3$ root system,
under the bijective correspondence
between lines in $\PP^2_\RR$ and planes through the origin in $\RR^3$.) 
It is a simplicial real arrangement. (This means that the lines divide the real projective plane into triangles.)
It is extremal. (If $t_k$
denotes the number of points where exactly $k$ lines meet, an inequality of Melchior \cite{Me}
for real arrangements of $d>2$ lines with $t_d=0$ states that $t_2\geq 3+\sum_{k>2}(k-3)t_k$.
For $B_3$ we have $t_2=6$, $t_3=4$ and $t_4=3$, so equality holds.)
It is free (meaning that if $F$ is the product of the linear forms defining the lines dual to the points of $Z$,
then there are no second syzygies for the Jacobian ideal $J_F=(F_x,F_y,F_z)$; i.e.,
the syzygy bundle for $J_F$ is free.)
Its unexpected curve has minimal degree in characteristic 0 (no unexpected curve 
in characteristic 0 has degree 3 or less \cite{A, FGST}).
And it is a supersolvable arrangement. 

A line arrangement in the projective plane is {\em supersolvable} if there is a so-called 
{\em modular} point, i.e., a point $P$ where two or more of the lines meet
such that if $Q$ is any other point where two or more of the lines meet, then the line through
$P$ and $Q$ is a line in the arrangement. 
Thus a supersolvable line arrangement includes the cone over its
crossing points with vertex at any modular point.
The {\em multiplicity} of a point with respect to a line arrangement
is just the number of lines in the arrangement containing the point. When a line arrangement is supersolvable,
every point of maximum multiplicity is modular (but not every modular point 
need have maximum multiplicity)  \cite{AT} (only the arXiv version includes the proof). For the line arrangement $B_3$, shown in Figure \ref{B3Fig}, the center point 
is modular and indeed has maximum multiplicity (no other crossing point has multiplicity more than 4).
The result of \cite{DMO} says the point scheme $Z_A$ dual to the lines of a supersolvable line arrangement $A$ 
has an unexpected curve of degree $m_A$ with respect to $X=(m_A-1)P$ for a general point $P$ 
if and only if $2m_A<d_A$, where $d_A$ is the number 
of lines in the arrangement and $m_A$ is the maximum multiplicity of a point for $A$, 
and in this case the unexpected curve is unique. 
Since $m_{B_3}=4$ and $d_{B_3}=9$, it follows that $Z_{B_3}$
has a unique unexpected curve of degree 4.

The roots of $B_3$ can be defined as the integer vector solutions $(a,b,c)\in\RR^3$
to $1\leq a^2+b^2+c^2\leq2$.
Geometrically, given a cube of side length $2$ aligned with the coordinate axes of
$\RR^3$ and whose center is at the origin
(the cube with vertices $(\pm 1, \pm 1, \pm 1)$), 
these are the vectors pointing from the origin
to the center of each face and to the midpoint of each edge. The roots in pairs correspond to points
in the projectivization $\PP^2_\RR$ of $\RR^3$. Thus the 18 roots give the 9 points
of $Z_{B_3}$, and the lines of the line arrangement $B_3$ are just the projectivizations
of the planes normal to the roots; these lines are the projective duals of the points of $Z_{B_3}$.

\begin{figure}
\begin{center}
\begin{tikzpicture}[line cap=round,line join=round,>=triangle 45,x=.5cm,y=.5cm]
\clip(-3,-3) rectangle (3,3);
\draw [line width=1.pt] (0.,-4.7) -- (0.,4.18);
\draw [line width=1.pt,domain=-4.98:6.62] plot(\x,{(-0.-0.*\x)/1.});
\draw [line width=1.pt,domain=-4.98:6.62] plot(\x,{(-0.-1.*\x)/1.});
\draw [line width=1.pt,domain=-4.98:6.62] plot(\x,{(-0.-1.*\x)/-1.});
\draw [line width=1.pt,domain=-4.98:6.62] plot(\x,{(-1.-1.*\x)/1.});
\draw [line width=1.pt,domain=-4.98:6.62] plot(\x,{(-1.-1.*\x)/-1.});
\draw [line width=1.pt,domain=-4.98:6.62] plot(\x,{(-1.--1.*\x)/-1.});
\draw [line width=1.pt,domain=-4.98:6.62] plot(\x,{(-1.--1.*\x)/1.});
\end{tikzpicture}
\caption{The nine lines of the arrangement $B_3$ (the line $z=0$ at infinity is not shown).}
\label{B3Fig}
\end{center}
\end{figure}
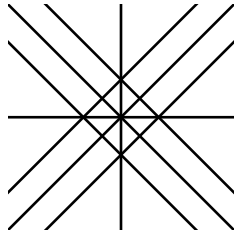

Given the interesting behavior of $B_3$, it is natural to look at other arrangements with similar properties.
As noted above, \cite{DMO} has done this for the case of supersolvable arrangements.
Here we check what happens for arrangements coming from other root systems
$A$ in $\RR^n$, not only for $n=3$ but also for $n>3$.
The set-up then is:
a root system $A$ gives a finite set of vectors of $\RR^n$ for some $n$.
Each root gives a point in $\PP^{n-1}$ and the set of these points for the given root system $A$
gives the point set we denote by $Z_A$. The codimension 1 linear subspaces normal to the
roots define the hyperplanes of the arrangement corresponding to $A$ which we also refer to by $A$.

In principle, given a finite set of points $Z\subset\PP^n$ and a general point $P=[a_0:\cdots:a_n]$,
to find an unexpected hypersurface for $Z+mP$ computationally one takes $P$ to be the generic point
$P=[1:\frac{a_1}{a_0}:\cdots:\frac{a_n}{a_0}]$ and works as usual in the homogeneous coordinate ring
$S=\KK(\frac{a_1}{a_0},\ldots,\frac{a_n}{a_0})[x_0,\ldots,x_n]=
\KK[\PP^n_{\KK(\frac{a_1}{a_0},\ldots,\frac{a_n}{a_0})}]$. 
(We will however abuse notation and typically use $[a_0:\cdots:a_n]$ to denote $P$
and refer to it as a general point.) However, it can be convenient
and more efficient to work in the bi-graded ring $R=\KK[a_0,\ldots,a_n][x_0,\ldots,x_n]$. So we mention here a few 
clarifying but elementary remarks.

All of our rings of interest are contained in the field $\KK(a_0,\ldots,a_n, x_0,\ldots,x_n)$, which is the field of fractions
of a UFD. Thus given a form $F\in S$ of degree $d$ but which is not in $\KK$,
there is (up to scalars in $\KK$) a unique factorization $F=BG/H$, where $B$ and $H$ are relatively prime forms in $\KK[a_0,\ldots,a_n]$ and 
$G\in R$ is bi-homogeneous with no factors of bi-degree $(a,0)$ with $a>0$; its bi-degree is $(t,d)$, where $t=\deg(H)-\deg(B)$. 
We denote $G$ by $F^*$. Given any bi-homogeneous element $G\in R$ of bi-degree $(t,d)$,
we denote $G/a_0^t$ by $G^\circ$. Note that $G^\circ\in S$ is homogeneous
of degree $d$. It need not be true that $(G^\circ)^*=G$ (for example, $(a_1^\circ)^*=(a_1/a_0)^*=1$) nor that
$(F^*)^\circ=F$ (e.g., $((a_1x_1/a_0)^*)^\circ=x_1^\circ=x_1$), but we do have the following useful lemma.

\begin{lemma}\label{StEx} 
Let $S$ and $R$ be as above.
Let $P$ be the point $[1:\frac{a_1}{a_0}:\cdots:\frac{a_n}{a_0}] \in\PP^n_{\KK(\frac{a_1}{a_0},\ldots,\frac{a_n}{a_0})}$.
For any $m \geq 1$, let $I_{mP}$ denote the ideal $(I_P)^m$ in $S$,
and let $J_{mP}$ denote the ideal $(J_P)^m$ in $R$,
where $J_P=(\{a_ix_j-a_jx_i: i\neq j\})$ (hence $(J_P)^m$ is generated by elements of bi-degree $(m,m)$).
Now let $d>0$ and $t \geq 0$, let $F\in S$ be a nonzero form of degree $d$,
and let $G\in R$ be a nonzero, nonconstant bi-homogeneous form of bi-degree $(t,d)$.

\begin{enumerate}
\item[(a)] As ideals in $S$ we have $(F)=((F^*)^\circ)$. 
\item[(b)] If $F$ is irreducible, then so are $F^*$ and $(F^*)^\circ$.
\item[(c)] If $G$ is irreducible, then so are $G^\circ$ and $(G^\circ)^*$, and as ideals in $R$ we have $(G)=((G^\circ)^*)$. 
\item[(d)] If $F\in I_{mP}$, then $d\geq m$ and 
$F^*\in J_{mP}$ so we see that $F^*$ has bi-degree $(s,d)$ for some $s$ with $s\geq m$. 
\item[(e)] If $G\in J_{mP}$, then $G^\circ\in I_{mP}$.
\end{enumerate}
\end{lemma}

\begin{proof}
(a) Since $F=BF^*/H=B(F^*)^\circ a_0^t/H$, where $t=\deg(H)-\deg(B)$, we see that
$F$ and $(F^*)^\circ$ differ by $Ba_0^t/H$, which is a unit in $S$.

(b) Note that $F$ and $(F^*)^\circ$ differ by a unit factor in $S$, so one is irreducible if and only if the other is.
Also, $F^*$ by construction has no factors of bi-degree $(f,0)$ with $f>0$, so if $F^*$ fails to be irreducible, it factors as
$F^*=AB$ where $A$ and $B$ have bi-degree $(a,d_A)$ and $(b,d_B)$ with $d_A,d_B>0$ and $d_A+d_B=d$.
But then $(F^*)^\circ=A^\circ B^\circ$ is a product of factors of positive degree and hence not irreducible.

(c) If $G^\circ$ is not irreducible, then $G^\circ=AB$ where both $A$ and $B$ have positive degree $d_A$ and $d_B$, so we have
$G=\alpha A^*\beta B^*$ where $A^*$ and $B^*$ have bi-degrees $(a,d_A)$ and $(b,d_B)$, and $\alpha$ and $\beta$ have bi-degrees
$(t_\alpha,0)$ and $(t_\beta,0)$. Since $R$ is a UFD, the denominators in the right hand side of the expression $G=\alpha A^*\beta B^*$ must
cancel with factors of  the numerators of the expression $\alpha A^*\beta B^*$ and this does not affect the values of $d_A$ or $d_B$, 
so we may assume that $\alpha,\beta, A^*, B^*$ all are in $R$ and $A^*$ and $B^*$ have bi-degrees $(a',d_A)$ and $(b',d_B)$,
where the simplification might have changed $a$ and $b$ but will have left $d_a$ and $d_B$ unchanged. Since $d_A$ and $d_B$ are both positive,
$G$ has nonunit factors so $G$ is not irreducible.
Now note that irreducibility of $G^\circ$ implies that for $(G^\circ)^*$ by (b). Finally, by construction, if $G$ is irreducible with bi-degree $(t,d)$ where $d>0$, then
$G=c(G^\circ)^*$ for some nonzero $c\in\KK$, hence $(G)=((G^\circ)^*)$.

(d) Since $F\in I_{mP}$, by (a) we have $\frac{F^*}{a_0^s}=(F^*)^\circ\in I_{mP}$, where $(s,d)$ is the bi-degree of $F^*$,
and since $d$ is the degree of $(F^*)^\circ$ and since $(F^*)^\circ\in I_{mP}$ we have $d\geq m$. 
Now let 
$$D\big(\frac{a_1}{a_0},\ldots,\frac{a_n}{a_0},\frac{x_1}{x_0},\ldots,\frac{x_n}{x_0}\big)=
F^*\big(\frac{a_0}{a_0},\ldots,\frac{a_n}{a_0},\frac{x_0}{x_0},\frac{x_1}{x_0},\ldots,\frac{x_n}{x_0}\big)=\frac{F^*(a_0,\ldots,a_n,x_0,\ldots,x_n)}{a_0^sx_0^d},$$
so as a rational function $D$ has bi-degree $(0,0)$, but as a polynomial in
$\KK(\frac{a_1}{a_0},\ldots,\frac{a_n}{a_0})[\frac{x_1}{x_0},\ldots,\frac{x_n}{x_0}]$, $D$ has degree $\delta\leq d$.
Moreover, since $x_0\not\in I_P$, we have $D\in I_{mQ}$ where $Q$ is the point $Q=(\frac{a_1}{a_0},\ldots,\frac{a_n}{a_0})$
in the affine open subset $\KK(\frac{a_1}{a_0},\ldots,\frac{a_n}{a_0})^n$ away from $x_0=0$. Now translate $Q$ to the origin; i.e., consider
$$H\big(\frac{a_1}{a_0},\ldots,\frac{a_n}{a_0},\frac{x_1}{x_0},\ldots,\frac{x_n}{x_0}\big)=
D\big(\frac{a_1}{a_0},\ldots,\frac{a_n}{a_0},\frac{x_1}{x_0}+\frac{a_1}{a_0},\ldots,\frac{x_n}{x_0}+\frac{a_n}{a_0}\big).$$
Note that 
$$a_0^{s+d}x_0^dH=a_0^{s+d}x_0^dF^*\big(\frac{a_0}{a_0},\ldots,\frac{a_n}{a_0},\frac{x_0}{x_0},\frac{x_1}{x_0}+\frac{a_1}{a_0},\ldots,\frac{x_n}{x_0}+\frac{a_n}{a_0}\big)=$$
$$F^*\big(a_0,\ldots,a_n,a_0x_0,a_0x_1+a_1x_0,\ldots,a_0x_n+a_nx_0\big)\in R$$
is a bi-homogeneous polynomial of bi-degree $(s+d,d)$.

Since $H$ has multiplicity $m$ at the origin (with respect to the variables $\frac{x_1}{x_0},\ldots,\frac{x_n}{x_0}$),
$H$ is a sum of terms where each term consists of a monomial in $\frac{x_1}{x_0},\ldots,\frac{x_n}{x_0}$ of degree at least $m$ and at most $d$,
times a polynomial in the variables $\frac{a_1}{a_0},\ldots,\frac{a_n}{a_0}$ of degree at most $s+d$. Thus each term is of the form
$$c\big(\frac{a_1}{a_0},\ldots,\frac{a_n}{a_0}\big)\big(\frac{x_1}{x_0}\big)^{i_1}\cdots\big(\frac{x_n}{x_0}\big)^{i_n}$$
where $m\leq \sum_ji_j\leq d$ and $c$ has degree at most $s+d$. Thus multiplying by $a_0^{s+d}x_0^d$ clears the denominators.
Translating back we recover $D$; i.e., 
$$D\big(\frac{a_1}{a_0},\ldots,\frac{a_n}{a_0},\frac{x_1}{x_0},\ldots,\frac{x_n}{x_0}\big)=
H\big(\frac{a_1}{a_0},\ldots,\frac{a_n}{a_0},\frac{x_1}{x_0}-\frac{a_1}{a_0},\ldots,\frac{x_n}{x_0}-\frac{a_n}{a_0}\big),$$
but each term becomes
$$c\big(\frac{a_1}{a_0},\ldots,\frac{a_n}{a_0}\big)\big(\frac{x_1}{x_0}-\frac{a_1}{a_0}\big)^{i_1}\cdots\big(\frac{x_n}{x_0}-\frac{a_n}{a_0}\big)^{i_n}
=c\big(\frac{a_1}{a_0},\ldots,\frac{a_n}{a_0}\big)\big(\frac{a_0x_1-a_1x_0}{a_0x_0}\big)^{i_1}\cdots\big(\frac{a_0x_n-a_nx_0}{a_0x_0}\big)^{i_n}$$
and multiplying by $a_0^{2d+s}x_0^d$ we obtain
$$C(a_0,\ldots,a_n)(a_0x_1-a_1x_0)^{i_1}\cdots(a_0x_n-a_nx_0)^{i_n}$$
where $C$ is homogeneous in the variables $a_i$ of degree $s+d$, and 
$$(a_0x_1-a_1x_0)^{i_1}\cdots(a_0x_n-a_nx_0)^{i_n}\in J_{mP}.$$
Thus $a_0^{2d+s}x_0^dH\big(\frac{a_1}{a_0},\ldots,\frac{a_n}{a_0},\frac{x_1}{x_0}-\frac{a_1}{a_0},\ldots,\frac{x_n}{x_0}-\frac{a_n}{a_0}\big)=
a_0^{2d+s}x_0^dD=a_0^{2d}F^*\in J_{mP}$. But by \cite[Corollary 7.10]{BV}, since $J_P$ is the ideal generated by the 
maximal minors of a $2 \times (n+1)$ matrix of indeterminates, namely the two sets of variables, the powers $J_{mP}$ are primary for $J_P$.
Since no power of $a_0$ is in $J_{mP}$, we see that $F^*\in J_{mP}$.

(e) If $G\in J_{mP}$, then $G=\sum_s H_sM_s$ where each $M_s$ is a product of $m$ forms such as $a_ix_j-a_jx_i$ and each $H_s$
is a bi-homogeneous form of bi-degree $(t-m,d-m)$. So $G^\circ$ is obtained by dividing $G$ by $a_0^t$, hence we get 
$G^\circ=\sum_s H_s^\circ M_s^\circ$, and each $M_s^\circ$ is a product of $m$ forms such as $(a_ix_j-a_jx_i)/a_0$, each of
which is in $I_P$. Thus $G^\circ\in I_{mP}$.

\end{proof}

We used Macaulay2 \cite{M2} to find and verify the occurrence of unexpected hypersurfaces for various root systems.
It seems likely that the root system $B_{n+1}$, which gives a point set $Z_{B_{n+1}}$ in $\PP^n$, gives rise to an unexpected hypersurface
for each $n\geq 2$, but our verifications of unexpectedness are computational and so are limited to a few smaller values of $n$. 
We do not have a general proof. 

Given a finite set of points $Z\subset\PP^n$ we use the following script to check whether $Z$ has an
unexpected hypersurface of degree $d$ vanishing on $mP$ for a general point $P=[a_0:\cdots:a_n]$. 
Assuming that the computation of rank given by Macaulay2 is reliable,
running the script on an example gives a rigorous proof of whether the example is unexpected or not. 

The idea of the script is to construct the matrix $N$ expressing the conditions imposed on
all forms $F$ of degree $d$ by the points of $Z$, together with the conditions imposed for $F$ to vanish to order $m$ at a point
$P=[a_0:\ldots:a_n]$ with indeterminate coordinates $a_i$ represented in the scripts with new variables. 
Thus if we enumerate the monomials of degree $d$ in $n+1$ variables $x_i$
as $b_1,\ldots,b_{\binom{n+d}{n}}$, then $N$ will be a matrix with $|Z|+\binom{n+m-1}{n}$ rows and $\binom{n+d}{n}$ columns.
A form $F=\sum_ic_ib_i$ vanishes on $Z$ and on $mP$ if and only if $Nc=0$, where $c$ is the coefficient vector $c=(c_1,\ldots,c_{\binom{n+d}{n}})^T$.
We can regard $N$ as consisting of two matrices, $Q_1$ and $Q_2$, where $Q_1$ comprises the top $|Z|$ rows of $N$ and gives the conditions
imposed by the points of $Z$, and $Q_2$ comprises the bottom $\binom{n+m-1}{n}$ rows of $N$ giving the conditions imposed
by the fat point $mP$. Thus the entries of $Q_1$ are scalars in the ground field, but the entries of $Q_2$ are the order $m-1$ partials of
the monomials $b_i$ evaluated at $P$. 
Then $Z$ has unexpected hypersurfaces of degree $d$ with a general point $P$ of multiplicity $m$
exactly when 
$$\dim \ker N > \max(0,\binom{n+d}{n}-\rank(Q_1)-\rank(Q_2)),$$
and in this case the coefficient vectors $c$
of the unexpected hypersurfaces are precisely the nontrivial elements of $\ker N$.

The script which does this is given below. The section marked CODE BLOCK is where one puts in the list of the points
of $Z$ (or where one puts in code needed to generate the list; see the examples).
For now we exhibit code which works when the points of $Z$ are defined over the rationals.
Subtleties arise when coordinates in an extension field are needed. We discuss that later.

Apart from output indicating current status of the computation,
the script output indicates exactly when it finds an unexpected hypersurface.
An output of the form $(n,d,m,edim,adim)$ means that there is
a hypersurface for $Z$ in $\PP^n$ of degree $d$ with a generic point
of multiplicity $m$ (that is a point whose coordinates are variables); the vector space of unexpected forms has expected dimension edim
and actual dimension adim (where edim can be negative if the number of conditions
imposed by $mp$ is greater than the dimension of $[I_Z]_d$).
If for a given $n,d,m$ there is no output, then $Z$ has no unexpected hypersurface
in $\PP^n$ of degree $d$ with a general point $P$ of multiplicity $m$.
We will refer to the script below as the {\it universal script}.

\begin{verbatim}
for n from 2 to 6 do { -- Loop over dim 2 to 6
R=QQ[x_0..x_n];      -- Define coordinate ring
S=frac(QQ[a_0..a_n]); --Define ring for generic point 
CODE BLOCK for Pts={...}; -- Insert the list of the points of Z here
print {"n=",n,"#Pts=",#Pts};  -- Print completion status indicator
for d from 2 to 6 do {   -- Loop over degree
Md=flatten entries(basis(d,R));  -- Create deg d monomial basis
for m from 2 to d do {   -- Loop over multiplicity m
print {"d=",d,"m=",m};    -- Print completion status indicator
Mm=flatten entries(basis(m-1,R));  -- Create deg m-1 monomial basis
A={};   -- A will contain the list of rows for transpose of matrix Q1
apply(Md,i->(N={};for s from 0 to #Pts-1 do N=N|{sub(i,matrix{Pts_s})};A=A|{N}));
D={};  -- D will contain the list of rows for transpose of matrix Q2
apply(Md,i->(N={};for s from 0 to #Mm-1 do N=N|{diff(Mm_s,i)};D=D|{N}));
Q1=transpose matrix A;  -- Q1 is defined over R
Q2=transpose matrix D;  -- Q2 is defined over R
M={};
for i from 0 to n do M=M|{a_i};  -- M is the coord vector for generic point
Q1S=sub(Q1,S); -- Q1 is now defined to be over S
Q2S=sub(Q2,matrix{M}); -- Swap x variables for a variables
N=Q1S||Q2S;  
expdim=#Md - (rank Q1S) - (rank Q2S);
actdim=#Md-(rank N);
if actdim > expdim and actdim > 0 then 
     print {"n=",n,"d=",d,"m=",m,"edim=",expdim,"adim=",actdim}}}}
\end{verbatim}

One can recover the actual unexpected forms by putting in a line to print out the kernel of $N$.
When the actual unexpected forms themselves are not needed,
the script can be made more efficient. Note that the line \verb!Q2S=sub(Q2T,matrix{M})!
merely substitutes the variables $a_i$ in for the variables $x_i$ in the $(m-1)$ order partials.
This doesn't affect the rank. Also, the rank of the matrix over $S$ after this substitution
is the same as the rank of the matrix over $R$ before the substitution.
Thus if existence of and numerical data for unexpectedness is all that is needed,
then the line \verb!S=frac(QQ[a_0..a_n]);! can be deleted and the lines
\begin{verbatim}
M={};
for i from 0 to n do M=M|{a_i};  -- M is the coord vector for generic point
Q1S=sub(Q1,S); -- Q1 is now defined to be over S
Q2S=sub(Q2,matrix{M}); -- Swap x variables for a variables
N=Q1S||Q2S;  
expdim=#Md - (rank Q1S) - (rank Q2S);
actdim=#Md-(rank N);
if actdim > expdim and actdim > 0 then 
     print {"n=",n,"d=",d,"m=",m,"edim=",expdim,"adim=",actdim}}}}
\end{verbatim}
can be changed to
\begin{verbatim}
N=Q1||Q2;
expdim=#Md - (rank Q1) - (rank Q2);
actdim=#Md-(rank N);
if actdim > expdim and actdim > 0 then 
     print {"n=",n,"d=",d,"m=",m,"edim=",expdim,"adim=",actdim}}}}
\end{verbatim}
It's possible the script would run faster by evaluating the matrix $Q_2$ of partials at a random point
(with coordinates in the rationals or even in the integers)
rather than at a generic point. To do so, replace the line\\
\verb!apply(Md,i->(N={};for s from 0 to #Mm-1 do N=N|{diff(Mm_s,i)};D=D|{N}));!\\
with
\begin{verbatim}
G={};
v={};
F=random(1,R);
for i from 0 to n do (v=v|{diff(x_i,F)});
G=G|{v};
apply(Md,i->(N={};for s from 0 to #Mm-1 do 
     N=N|{sub(diff(Mm_s,i),matrix G)};D=D|{N}));
\end{verbatim}
By semicontinuity, if the script indicates that there is no
unexpected hypersurface for a given $Z$, $n$, $d$ and $m$ (by not outputting anything),
then there is indeed no such unexpected hypersurface. But output claiming
an unexpected hypersurface can't be relied on since the random point might
have been unlucky.

We now check the usual root systems for occurrence of unexpected hypersurfaces.

\subsection{The root system \texorpdfstring{$A_{n+1}$}{A(n+1)}}
The roots for $A_{n+1}$ are the $(n+1)(n+2)$ integer vectors in $\RR^{n+2}$ having one entry of 1, one of $-1$ and the rest 0.
We project these into $\RR^{n+1}$ by dropping the last coordinate. Projectivizing then gives a set $Z\subset\PP^n$
of $\binom{n+1}{2}$ points. No unexpected hypersurfaces turned up for 
$2\leq n\leq6$, $2\leq d\leq6$, $2\leq m\leq d$.

\subsection{The root system \texorpdfstring{$B_{n+1}$}{Bn}}
The root system $B_{n+1}\subset\RR^{n+1}$ consists of the $2(n+1)^2$ integer vectors $(a_1,\ldots,a_{n+1})$ 
such that $1\leq a_1^2+\cdots+a_{n+1}^2\leq 2$.
We thus have $|Z_{B_{n+1}}|=(n+1)^2$ for the corresponding set of points $Z_{B_{n+1}}\subset\PP^n$.
The CODE BLOCK here is:
\begin{verbatim}
H={};
W1=subsets(n+1,1);
apply(W1,s->(H=H|{x_(s_0)}));
W2=subsets(n+1,2);
apply(W2,s->(H=H|{x_(s_0)+x_(s_1),x_(s_0)-x_(s_1)}));
Pts={};
for j from 0 to #H-1 do (v={};for i from 0 to n do 
     (v=v|{diff(x_i,H_j)});Pts=Pts|{v});
\end{verbatim}
Checking $2\leq n\leq6$, $2\leq d\leq6$, $2\leq m\leq d$ turned up
seven cases of unexpected hypersurfaces, listed in Table~\ref{table: unexpected hypersurfaces}.

\begin{table}
\begin{tabular}{rrrrr}
$n$ & $d$ & $m$ & edim & adim \\
\hline
$2$ & $4$ & $3$ & $0$ & $1$ \\
$3$ & $4$ & $4$ & $-1$ & $1$ \\
$4$ & $4$ & $4$ & $10$ & $11$ \\
$5$ & $3$ & $3$ & $-1$ & $5$ \\
$5$ & $4$ & $4$ & $34$ & $35$ \\
$6$ & $3$ & $3$ & $7$ & $14$ \\
$6$ & $4$ & $4$ & $77$ & $78$ \\
\end{tabular}
\caption{Unexpected hypersurfaces arising from the root system $B_{n+1}$.}\label{table: unexpected hypersurfaces}
\end{table}

The case $(2, 4, 3, 0, 1)$ comes from the arrangement $B_3$ shown in Figure \ref{B3Fig}.
Its unique unexpected curve was shown to be unexpected by other methods in \cite{CHMN}.

In the case of $B_4$ we get a previously unknown unique unexpected hypersurface. It has degree 4 with a general point $[a_0:a_1:a_2:a_3]$ of multiplicity 4.
Thus it is a cone at the point $[a_0:a_1:a_2:a_3]$. In this case $Z_{B_4}$ is the set of 16 points coming from the roots, but the vanishing locus of $[I_{Z_{B_4}}]_4$
is 0-dimensional, in contrast with the examples of \S \ref{cone section}. In fact, as the general point of multiplicity 4 moves around, the only points that lie on every
one of the degree 4 unexpected surfaces are the 16 points of $Z_{B_4}$ together with the eight points
$[1:1:1:1]$,
$[-1:1:1:1]$,
$[1:-1:1:1]$,
$[1:1:-1:1]$,
$[1:1:1:-1]$,
$[-1:-1:1:1]$,
$[-1:1:-1:1]$,
$[1:-1:-1:1]$.
(To compute this locus of 24 points, look at the ideal of the coefficients of the monomials of the unexpected form
$F(a,x)$ in $a=[a_0:a_1:a_2:a_3]$; i.e., if $F(a,x)$ is the unexpected surface,
write it as a polynomial in $a_i$ with coefficients which are polynomials in $x_j$. Take the ideal generated by these 
coefficient polynomials in $x_j$. They define the locus of points at which all of the unexpected surfaces vanish as the point
$a$ moves around.)

\subsection{The root system \texorpdfstring{$C_{n+1}$}{C(n+1)}}
Since $Z_{C_{n+1}}=Z_{B_{n+1}}$, this case is covered by $B_{n+1}$.

\subsection{The root system \texorpdfstring{$D_{n+1}$}{D(n+1)}}
The root system $D_{n+1}$ consists of the $2(n+1)n$ integer vectors $(a_1,\ldots,a_{n+1})\in\RR^{n+1}$ with $a_1^2+\cdots+a_{n+1}^2=2$.
The CODE BLOCK here is:
\begin{verbatim}
H={};
W2=subsets(n+1,2);
apply(W2,s->(H=H|{x_(s_0)+x_(s_1),x_(s_0)-x_(s_1)}));
Pts={};
for j from 0 to #H-1 do 
    (v={}; for i from 0 to n do (v=v|{diff(x_i,H_j)});Pts=Pts|{v});
\end{verbatim}
We checked $2\leq n\leq6$, $2\leq d\leq 6$ and $2\leq m\leq d$. The cases $({n, d, m, edim, adim})$ 
that turned up were $({3, 3, 3, -2, 1}), ({3, 4, 4, 3, 4})$.

\subsection{The root systems \texorpdfstring{$E_{n+1}$}{E(n+1)}, for \texorpdfstring{$n=5,6,7$}{n=5,6,7}}
The root system $E_8$ is a set of 240 vectors consisting of all integer vectors $(a_1,\ldots,a_8)\in\RR^8$ with $a_1^2+\cdots+a_8^2=2$,
(i.e., $D_8$) together with all vectors of the form $(1/2)(a_1,\ldots,a_8)$ where each entry $a_i$ is $\pm1$ and $a_1\cdots a_8=1$;
thus $|Z_{E_8}|=120$.
The CODE BLOCK here is:
{\tiny
\begin{verbatim}
Pts={{1, 1, 0, 0, 0, 0, 0, 0}, {1, -1, 0, 0, 0, 0, 0, 0}, {1, 0, 1, 0, 0, 0, 0, 0}, 
{1, 0, -1, 0, 0, 0, 0, 0}, {0, 1, 1, 0, 0, 0, 0, 0}, {0, 1, -1, 0, 0, 0, 0, 0}, 
{1, 0, 0, 1, 0, 0, 0, 0}, {1, 0, 0, -1, 0, 0, 0, 0}, {0, 1, 0, 1, 0, 0, 0, 0}, 
{0, 1, 0, -1, 0, 0, 0, 0}, {0, 0, 1, 1, 0, 0, 0, 0}, {0, 0, 1, -1, 0, 0, 0, 0}, 
{1, 0, 0, 0, 1, 0, 0, 0}, {1, 0, 0, 0, -1, 0, 0, 0}, {0, 1, 0, 0, 1, 0, 0, 0}, 
{0, 1, 0, 0, -1, 0, 0, 0}, {0, 0, 1, 0, 1, 0, 0, 0}, {0, 0, 1, 0, -1, 0, 0, 0}, 
{0, 0, 0, 1, 1, 0, 0, 0}, {0, 0, 0, 1, -1, 0, 0, 0}, {1, 0, 0, 0, 0, 1, 0, 0}, 
{1, 0, 0, 0, 0, -1, 0, 0}, {0, 1, 0, 0, 0, 1, 0, 0}, {0, 1, 0, 0, 0, -1, 0, 0}, 
{0, 0, 1, 0, 0, 1, 0, 0}, {0, 0, 1, 0, 0, -1, 0, 0}, {0, 0, 0, 1, 0, 1, 0, 0}, 
{0, 0, 0, 1, 0, -1, 0, 0}, {0, 0, 0, 0, 1, 1, 0, 0}, {0, 0, 0, 0, 1, -1, 0, 0}, 
{1, 0, 0, 0, 0, 0, 1, 0}, {1, 0, 0, 0, 0, 0, -1, 0}, {0, 1, 0, 0, 0, 0, 1, 0}, 
{0, 1, 0, 0, 0, 0, -1, 0}, {0, 0, 1, 0, 0, 0, 1, 0}, {0, 0, 1, 0, 0, 0, -1, 0}, 
{0, 0, 0, 1, 0, 0, 1, 0}, {0, 0, 0, 1, 0, 0, -1, 0}, {0, 0, 0, 0, 1, 0, 1, 0}, 
{0, 0, 0, 0, 1, 0, -1, 0}, {0, 0, 0, 0, 0, 1, 1, 0}, {0, 0, 0, 0, 0, 1, -1, 0}, 
{1, 0, 0, 0, 0, 0, 0, 1}, {1, 0, 0, 0, 0, 0, 0, -1}, {0, 1, 0, 0, 0, 0, 0, 1}, 
{0, 1, 0, 0, 0, 0, 0, -1}, {0, 0, 1, 0, 0, 0, 0, 1}, {0, 0, 1, 0, 0, 0, 0, -1}, 
{0, 0, 0, 1, 0, 0, 0, 1}, {0, 0, 0, 1, 0, 0, 0, -1}, {0, 0, 0, 0, 1, 0, 0, 1}, 
{0, 0, 0, 0, 1, 0, 0, -1}, {0, 0, 0, 0, 0, 1, 0, 1}, {0, 0, 0, 0, 0, 1, 0, -1}, 
{0, 0, 0, 0, 0, 0, 1, 1}, {0, 0, 0, 0, 0, 0, 1, -1}, {1, 1, 1, 1, 1, 1, 1, 1}, 
{1, 1, 1, 1, 1, 1, -1, -1}, {1, 1, 1, 1, 1, -1, 1, -1}, {1, 1, 1, 1, 1, -1, -1, 1}, 
{1, 1, 1, 1, -1, 1, 1, -1}, {1, 1, 1, 1, -1, 1, -1, 1}, {1, 1, 1, 1, -1, -1, 1, 1}, 
{1, 1, 1, 1, -1, -1, -1, -1}, {1, 1, 1, -1, 1, 1, 1, -1}, {1, 1, 1, -1, 1, 1, -1, 1}, 
{1, 1, 1, -1, 1, -1, 1, 1}, {1, 1, 1, -1, 1, -1, -1, -1}, {1, 1, 1, -1, -1, 1, 1, 1}, 
{1, 1, 1, -1, -1, 1, -1, -1}, {1, 1, 1, -1, -1, -1, 1, -1}, {1, 1, 1, -1, -1, -1, -1, 1}, 
{1, 1, -1, 1, 1, 1, 1, -1}, {1, 1, -1, 1, 1, 1, -1, 1}, {1, 1, -1, 1, 1, -1, 1, 1}, 
{1, 1, -1, 1, 1, -1, -1, -1}, {1, 1, -1, 1, -1, 1, 1, 1}, {1, 1, -1, 1, -1, 1, -1, -1}, 
{1, 1, -1, 1, -1, -1, 1, -1}, {1, 1, -1, 1, -1, -1, -1, 1}, {1, 1, -1, -1, 1, 1, 1, 1}, 
{1, 1, -1, -1, 1, 1, -1, -1}, {1, 1, -1, -1, 1, -1, 1, -1}, {1, 1, -1, -1, 1, -1, -1, 1}, 
{1, 1, -1, -1, -1, 1, 1, -1}, {1, 1, -1, -1, -1, 1, -1, 1}, {1, 1, -1, -1, -1, -1, 1, 1}, 
{1, 1, -1, -1, -1, -1, -1, -1}, {1, -1, 1, 1, 1, 1, 1, -1}, {1, -1, 1, 1, 1, 1, -1, 1}, 
{1, -1, 1, 1, 1, -1, 1, 1}, {1, -1, 1, 1, 1, -1, -1, -1}, {1, -1, 1, 1, -1, 1, 1, 1}, 
{1, -1, 1, 1, -1, 1, -1, -1}, {1, -1, 1, 1, -1, -1, 1, -1}, {1, -1, 1, 1, -1, -1, -1, 1}, 
{1, -1, 1, -1, 1, 1, 1, 1}, {1, -1, 1, -1, 1, 1, -1, -1}, {1, -1, 1, -1, 1, -1, 1, -1}, 
{1, -1, 1, -1, 1, -1, -1, 1}, {1, -1, 1, -1, -1, 1, 1, -1}, {1, -1, 1, -1, -1, 1, -1, 1}, 
{1, -1, 1, -1, -1, -1, 1, 1}, {1, -1, 1, -1, -1, -1, -1, -1}, {1, -1, -1, 1, 1, 1, 1, 1}, 
{1, -1, -1, 1, 1, 1, -1, -1}, {1, -1, -1, 1, 1, -1, 1, -1}, {1, -1, -1, 1, 1, -1, -1, 1}, 
{1, -1, -1, 1, -1, 1, 1, -1}, {1, -1, -1, 1, -1, 1, -1, 1}, {1, -1, -1, 1, -1, -1, 1, 1}, 
{1, -1, -1, 1, -1, -1, -1, -1}, {1, -1, -1, -1, 1, 1, 1, -1}, {1, -1, -1, -1, 1, 1, -1, 1}, 
{1, -1, -1, -1, 1, -1, 1, 1}, {1, -1, -1, -1, 1, -1, -1, -1}, {1, -1, -1, -1, -1, 1, 1, 1}, 
{1, -1, -1, -1, -1, 1, -1, -1}, {1, -1, -1, -1, -1, -1, 1, -1}, {1, -1, -1, -1, -1, -1, -1, 1}};
\end{verbatim}
}

For $E_8$ we have $n=7$. 
We checked $2\leq d\leq 6$ and $2\leq m\leq d$. The cases $({n, d, m, edim, adim})$ 
that arose were
$({7, 4, 3, 174, 175})$,
$({7, 4, 4, 90, 99})$, and
$({7, 5, 5, 342, 343})$.

The root system $E_7$ is the set of 126 elements of $E_8$ such that the first two coordinates are equal.
Thus the CODE BLOCK for $E_7$ is obtained from that for $E_8$ by filtering out the cases where
the first two coordinates are equal and then dropping the first coordinate to get a 7 element vector.
We checked $2\leq n\leq6$, $2\leq d\leq 6$ and $2\leq m\leq d$. The only case $(n, d, m, edim, adim)$ 
that arose was $({6, 4, 4, 63, 64})$.

The root system $E_6$ is the set of 72 elements of $E_8$ such that the first three coordinates are equal.
The CODE BLOCK in this case is obtained from that for $E_7$ by filtering out the cases where
the first two coordinates are equal and then dropping the first coordinate to get a 6 element vector.
We checked $2\leq n\leq6$, $2\leq d\leq 6$ and $2\leq m\leq d$ but did not find any 
unexpected hypersurfaces.

\subsection{The root system \texorpdfstring{$F_4$}{F4}}
The root system $F_4$ is the set of 48 vectors consisting of all vectors $(a_1,\ldots,a_4)\in\RR^4$ 
such that each nonzero entry is $\pm1$ and $1\leq a_1^2+\cdots+a_4^2\leq2$,
or each entry is $\pm1/2$.
So here $n=3$, $|Z_{F_4}|=24$ and the CODE BLOCK is
\begin{verbatim}
Pts={{1,1,0,0}, {1,-1,0,0}, {1,0,1,0}, {1,0,-1,0}, {1,0,0,1}, {1,0,0,-1}, 
{0,1,1,0}, {0,1,-1,0}, {0,1,0,1}, {0,1,0,-1}, {0,0,1,1}, {0,0,1,-1}, 
{1,0,0,0}, {0,1,0,0}, {0,0,1,0}, {0,0,0,1}, {1,1,1,1}, {1,1,-1,1}, 
{1,1,1,-1}, {1,1,-1,-1}, {1,-1,1,1}, {1,-1,-1,1}, {1,-1,1,-1}, {1,-1,-1,-1}};
\end{verbatim}
We checked $2\leq n\leq6$, $2\leq d\leq 10$ and $2\leq m\leq d$. The cases $({n, d, m, edim, adim})$ 
that arose were
$({3, 4, 3, 2, 4}),
({3, 4, 4, -8, 1}),
({3, 5, 5, -3, 3}),
({3, 6, 6, 4, 7}),
({3, 7, 7, 12, 13})$.

\subsection{The root system \texorpdfstring{$H_n,n=3,4$}{Hn, n=3,4}}
The root systems $H_3$ and $H_4$ are non-crystallographic.
Figure \ref{H3Fig} shows the 15 lines dual to the roots of the $H_3$ root system.
We note that the $H_3$ line arrangement is simplicial but not supersolvable.
Because the lines are not defined over $\QQ$,
for $H_3$ we must in the universal script replace
\begin{verbatim}
for n from 2 to 2 do {
R=QQ[x_0..x_n]; 
S=frac(QQ[a_0..a_n]); 
\end{verbatim}
by
\begin{verbatim}
for n from 2 to 2 do {
K=toField(QQ[t]/(t^2-5));
R=K[x_0..x_n];
S=frac(QQ[a_0..a_n]); 
\end{verbatim}
\medskip
The CODE BLOCK is now
\begin{verbatim}
Pts={{0,0,1}, {1,0,1}, {1,0,-1}, {0,1,1},
{0,1,-1}, {0,1,2 + t}, {0,1,-2 - t}, {1,0,-2 - t},
{1,0,2 + t}, {1,-1,0}, {1,1,0},
{t+3,-(2*t+4),3*t+7}, {2*t+4,-(t+3),-(3*t+7)},
{2*t+4,-(t+3),3*t+7}, {t+3,-(2*t+4),-(3*t+7)}};
\end{verbatim}
The unexpected curves that turned up for $n=2$, $2\leq d\leq 8$ and $2\leq m\leq d$ were
$({2, 6, 5, -2, 1}),
({2, 7, 6, 0, 2}),
({2, 8, 7, 2, 3})$.

\begin{figure}
\begin{center}
\definecolor{cqcqcq}{rgb}{0.7529411764705882,0.7529411764705882,0.7529411764705882}
\begin{tikzpicture}[line cap=round,line join=round,>=triangle 45,x=.35cm,y=.35cm]
\clip(-9.004453999999999,-6.906597999999999) rectangle (7.979106000000004,6.82666);
\draw [line width=1.pt] (-1.,-6.906597999999999) -- (-1.,6.82666);
\draw [line width=1.pt] (1.,-6.906597999999999) -- (1.,6.82666);
\draw [line width=1.pt,domain=-9.004453999999999:7.979106000000004] plot(\x,{(-1.-0.*\x)/1.});
\draw [line width=1.pt,domain=-9.004453999999999:7.979106000000004] plot(\x,{(--1.-0.*\x)/1.});
\draw [line width=1.pt,domain=-9.004453999999999:7.979106000000004] plot(\x,{(-4.23606797749979-0.*\x)/1.});
\draw [line width=1.pt,domain=-9.004453999999999:7.979106000000004] plot(\x,{(--4.23606797749979-0.*\x)/1.});
\draw [line width=1.pt] (4.23606797749979,-6.906597999999999) -- (4.23606797749979,6.82666);
\draw [line width=1.pt] (-4.23606797749979,-6.906597999999999) -- (-4.23606797749979,6.82666);
\draw [line width=1.pt,domain=-9.004453999999999:7.979106000000004] plot(\x,{(-0.--1.*\x)/1.});
\draw [line width=1.pt,domain=-9.004453999999999:7.979106000000004] plot(\x,{(-0.-1.*\x)/1.});
\draw [line width=1.pt,domain=-9.004453999999999:7.979106000000004] plot(\x,{(-13.70820393249937-5.23606797749979*\x)/-8.47213595499958});
\draw [line width=1.pt,domain=-9.004453999999999:7.979106000000004] plot(\x,{(--13.70820393249937-8.47213595499958*\x)/-5.23606797749979});
\draw [line width=1.pt,domain=-9.004453999999999:7.979106000000004] plot(\x,{(-13.70820393249937-8.47213595499958*\x)/-5.23606797749979});
\draw [line width=1.pt,domain=-9.004453999999999:7.979106000000004] plot(\x,{(--13.70820393249937-5.23606797749979*\x)/-8.47213595499958});
\draw [fill=white] (0.,0.) circle (3.0pt);
\draw [fill=white] (4.23606797749979,4.23606797749979) circle (3.0pt);
\draw [fill=white] (1.,1.) circle (3.0pt);
\end{tikzpicture}
\caption{The $H_3$ configuration of 15 lines (the line at infinity is not shown; the affine coordinates of the points shown 
as open dots are $(0,0), (1,1)$ and $(2+\sqrt{5},2+\sqrt{5})$).}
\label{H3Fig}
\end{center}
\end{figure}
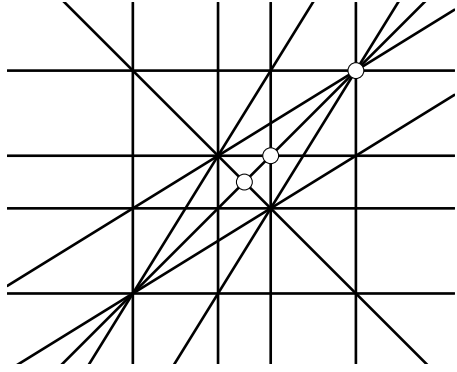

The $H_4$ root system has 120 elements defined over ${\QQ}[t]/(t^2-t-1)$
(see \cite{BFGMV,St}). Thus $|Z_{H_4}|=60$.
For $H_4$ one must use
\begin{verbatim}
for n from 3 to 3 do {
K=toField(QQ[t]/(t^2-t-1));
R=K[x_0..x_n];
\end{verbatim}
The CODE BLOCK is
{\tiny
\begin{verbatim}
Pts={{1,0,0,0}, {0,1,0,0}, {0,0,1,0}, {0,0,0,1}, {1,1,1,1}, {1,1,1,-1},
{1,1,-1,1}, {1,1,-1,-1}, {1,-1,1,1}, {1,-1,1,-1}, {1,-1,-1,1}, {1,-1,-1,-1},
{0,t,t^2,1}, {0,t,t^2,-1}, {0,t,-t^2,1}, {0,t,-t^2,-1}, {0,t^2,1,t}, {0,t^2,1,-t},
{0,t^2,-1,t}, {0,t^2,-1,-t}, {0,1,t,t^2}, {0,1,t,-t^2}, {0,1,-t,t^2}, {0,1,-t,-t^2},
{t,0,1,t^2}, {t,0,1,-t^2}, {t,0,-1,t^2}, {t,0,-1,-t^2}, {t^2,0,t,1}, {t^2,0,t,-1},
{t^2,0,-t,1}, {t^2,0,-t,-1}, {1,0,t^2,t}, {1,0,t^2,-t}, {1,0,-t^2,t}, {1,0,-t^2,-t},
{t,t^2,0,1}, {t,t^2,0,-1}, {t,-t^2,0,1}, {t,-t^2,0,-1}, {t^2,1,0,t}, {t^2,1,0,-t},
{t^2,-1,0,t}, {t^2,-1,0,-t}, {1,t,0,t^2}, {1,t,0,-t^2}, {1,-t,0,t^2}, {1,-t,0,-t^2},
{t,1,t^2,0}, {t,1,-t^2,0}, {t,-1,t^2,0}, {t,-1,-t^2,0}, {t^2,t,1,0}, {t^2,t,-1,0},
{t^2,-t,1,0}, {t^2,-t,-1,0}, {1,t^2,t,0}, {1,t^2,-t,0}, {1,-t^2,t,0}, {1,-t^2,-t,0}};
\end{verbatim}
}
The cases of unexpected surfaces that turned up for $n=3$, $2\leq d\leq 6$ and $2\leq m\leq d$ were
$({3, 6, 3, 14, 15})$,
$({3, 6, 4, 4, 9})$,
$({3, 6, 5, -11, 4})$,
$({3, 6, 6, -32, 1})$.


\section{BMSS Duality}\label{BMSSduality}

A very interesting observation was made in \cite{BMSS}. In the case of the unexpected quartic coming from the $B_3$ line arrangement,
\cite{BMSS} observed that $F^*(a,x)$ has bi-degree $(3,4)$ and for each $x$ it defines three lines 
in the $a$ variables, and moreover that these three lines meet at the point $x$.
In fact, given a form 
$$F\in S=\KK\big(\frac{a_1}{a_0},\ldots,\frac{a_n}{a_0}\big)[x_0,\ldots,x_n]$$
defining an unexpected variety of degree $d$ with a general point $P$ of multiplicity $m$,
one has by Lemma \ref{StEx}(d) that the bi-homogeneous form $F^*(a,x)\in R=\KK[a_0,\ldots,a_n][x_0,\ldots,x_n]$ has bi-degree $(t,d)$
for some $t\geq m$. Thus one can regard $F^*(a,x)$ as defining a family of hypersurfaces in the $x$ variables, parameterized by $a$
(these are the unexpected hypersurfaces), but one can also regard 
$F^*(a,x)$ as defining a more mysterious family of hypersurfaces in the $a$ variables, parameterized by $x$.

In the case of the $B_3$ unexpected quartic $F^*(a,x)$ with a general triple point, it follows from Lemma \ref{StEx}(d) that
$F^*$ has multiplicity at least 3 in the $a$ variables at the point $a=x$ and hence has bi-degree $(s,4)$ with $s\geq 3$.
We see below why in fact $s=3$; thus, given the point of multiplicity at least 3 at $a=x$,
it must have multiplicity exactly 3, and for a general choice of $x$, $F^*(a,x)$ splits as a product of three 
forms linear in the $a$ variables, meeting at $a=x$.

In this section we will show that this phenomenon occurs in a range of cases, 
and that for these cases $F^*(x,a)$ defines the lines tangent to the branches of the curve $F^*(a,x)$
(see Figure \ref{TCFig}). We will also study a similar duality for the hypersurfaces defined by our cone construction
from \S\ref{cone section}.

Let $W\subset \Ab^n$ be a hypersurface in affine space defined by a reduced polynomial $F$.
We first recall what the tangent cone is for $W$ at a point $P\in W$.
Write $F$ as a sum $F=F_0+F_1+\cdots$ of polynomials $F_i \in I(P)^i$ where each 
$F_i$ is homogeneous of degree $i$ in coordinates centered at the point $P$.
Let $j$ be the least index such that $F_j\neq0$. Then $F_j=0$ defines the tangent cone
to $W$ at $P$. In characteristic 0, $F=F_0+F_1+\cdots$ is just a Taylor expansion of $F$ at $P$; the tangent 
cone is the term $F_j$ obtained by differentiating to order $j$ where $j$ is the multiplicity of $W$ at $P$.

One can also work projectively. Let $W\subset \PP^n$ now be a hypersurface in projective space defined by a reduced form $H$
of degree $d$. For simplicity, assume the characteristic of $\KK$ is 0.
Given two polynomials $F$ and $G$, let $F\cdot G$ denote the action of differentiation, so 
$$(2x_0^2x_1+x_3)\cdot (x_0^3x_1^2x_3)=(2x_0^2x_1)\cdot (x_0^3x_1^2x_3)+x_3\cdot (x_0^3x_1^2x_3)=24x_0x_1x_3+x_0^3x_1^2.$$

To compute the tangent cone of $W$ at a point $P$ of multiplicity $m$,
let $\mu_j$ be an enumeration of the monomials of degree $m$. Let $c_j=\mu_j\cdot\mu_j$
be the factorial expression obtained by differentiating $\mu_j$ against itself (this is needed for the Taylor expansion).
Then the tangent cone of $W$ at $P$ is defined by the degree $m$ form
\begin{equation} \label{def of HP}
H_P=\sum_j \frac{((\mu_j\cdot H)(P))\mu_j}{c_j}.
\end{equation}

\begin{example}
As an example we consider the tangent cones for the cone construction of \S\ref{cone section}. 
Assume $Z$ is a finite set of points in $\PP^n_\KK$ and that the space $V\subset [S]_d$ of forms vanishing on the points
has the property that there is (up to multiplication by a scalar)
a unique form $F\in V$ with a point of multiplicity $d$ at a general point $P=[a_0:\ldots:a_n]\in\PP^n$. 
Applying an idea similar to that of Remark \ref{need fewer pts}, remove points if necessary so that we are 
left with a subset $\{P_1,\ldots,P_r\}$ of $Z$ of $r=\binom{d+n}{n} - \binom{d-1+n}{n} -1$ points. (Now $F$ is not unexpected.) 
Then $F^*\in R=\KK[a_0,\ldots,a_n][x_0,\ldots,x_n]$ and by Lemma \ref{StEx}(d) $F^*$ is bi-homogeneous
of bi-degree $(\delta,d)$ with $\delta\geq d$.

Let $M_j$ be an enumeration of the monomials in $T=\KK[x_0,\ldots,x_n]$ of degree $d$. Let $m_i$ be an enumeration
of the monomials in $T$ of degree $d-1$. Let $t=\binom{n+d}{n}$ and $s=\binom{n+d-1}{n}$.
Let $\Gamma$ be the matrix whose top $r$ rows are the values $M_j(P_i)$ of the monomials $M_j$ at the points $P_i$,
and whose next (and bottom) $s$ rows are the values $(m_i\cdot M_j)(P)$, 
where as above the dot indicates the action of $T$ on $T$ by partial differentiation.
Note that the entries of $\Gamma$ are all in $\KK[a_0,\ldots,a_n]$.
Elements in the kernel of $\Gamma$ are coefficient vectors for forms vanishing at the points $P_i$ and having a point of multiplicity $d$ at $P = a$.
The assumption that there are $r=t-s-1$ points $P_i$ and that there is (up to multiplication by scalars) a unique form vanishing on the points
with a point of multiplicity $m$ at $P$ means 
that $\Gamma$ is a $(t-1)\times t$ matrix whose rank at a general point $P$ is $t-1$. 
Since the entries of $\Gamma$ are monomials of degree 0 or 1, we can divide the $s$ rows
having degree 1 monomials by $a_0$ and obtain a row equivalent matrix $\Gamma'$
with entries in the field $\FF=\KK\big(\frac{a_1}{a_0},\ldots,\frac{a_n}{a_0}\big)$. The kernel of $\Gamma'$ has dimension 1, and
for any nonzero vector $v=(v_1,\ldots,v_t)$
in the kernel, we can take $F$ to be $F=\sum_iv_iM_i$. 
Note that $(F^*)^\circ$ (and hence $F^*$) also are in the kernel, and so can be used in place of $F$
and thus all have the same tangent cone for a general point $P=a$. 
Each entry of $v$ is in $\FF$ so computing the tangent cone gives
$$\sum_j \frac{((M_j\cdot F)(P))M_j}{c_j}=\sum_{i,j} \frac{((M_j\cdot v_iM_i)(P))M_j}{c_j}
=\sum_i v_i\frac{((M_i\cdot M_i)(P))M_i}{c_i}=\sum_i v_iM_i=F.$$
I.e., $F$ is its own tangent cone, which for a general point $P=a$ is thus defined by $F(a,x)=0$
(or equivalently $F^*(a,x)=0$).

We can also work over $R$.
Let $\Gamma'^*$ be the matrix obtained by appending $(M_1,\ldots,M_t)$ as a row at the bottom of $\Gamma'$,
and let $\Gamma^*$ be the matrix obtained by appending $(M_1,\ldots,M_t)$ as a row at the bottom of $\Gamma$.
Then $G'=\det\Gamma'^*$ is a multiple of $F$ by a scalar in $\FF$ and we have $a_0^sG'=G=\det \Gamma^*\in R$.
Since $G'=G/a_0^s$ is a scalar multiple of $F$ by a scalar in $\FF$, it follows that $G^*=F^*$. 
Moreover, we have $G=C(a)G^*$
where $C(a)$ is a polynomial describing for which points $a$ the matrix $\Gamma$ has less than full rank.
(Suppose that we choose a point $a=Q_1$ such that there is a point $x=Q_2$ for which 
$G^*(Q_1,Q_2)\neq0$. Then $C(Q_1)=0$ if and only if $\det \Gamma^*=G(Q_1,Q_2)=0$,
which occurs if and only if the maximal minors of $\Gamma$ all vanish; i.e., $\Gamma$ has rank less than $t-1$.)
As before, $G$ is its own tangent cone, but it's still not clear what $F^*$ is or how $F^*(a,x)$ is related to $F^*(x,a)$. 

Now assume that there is an irreducible variety $W$ of degree $d$ and codimension 2 such that for a general point 
$P=a$ the locus $F(a,x)=0$ is precisely the union of all lines through $P$ and a point of $W$,
and that $F(a,x)$ is irreducible (and hence so is $F^*(a,x)$ by Lemma \ref{StEx}). Then we have $F^*(a,x)=\pm F^*(x,a)$. Here's why.
For a general point $P'=[a_0':\ldots:a_n']$ of $F(a,x)=0$ (and hence of $F^*(a,x)=0$), $P'$ is on the cone through $W$ having vertex $P$,
so $P'$ is on the line through $P$ and a point $w\in W$. But then $P$ is on the line through $P'$ and $w$, so
$P$ is on the cone $F(a',x)=0$ with vertex $P'$ (hence on $F^*(a',x)=0$), so $F^*(a',a)=0$. I.e., $F^*(a',a)=0$ if and only if $F^*(a,a')=0$.
Thus the loci $F^*(a,x)=0$ and $F^*(x,a)=0$ intersect in a nonempty open subset of $F(a,x)=0$. 
Since $F^*(a,x)$ and $F^*(x,a)$ are irreducible, we have $F^*(a,x)=cF^*(x,a)$ for some scalar $c$.
But swapping variables again gives $F^*(a,x)=cF^*(a,x)$ hence $F^*(a,x)=c^2F^*(a,x)$ so $c=\pm1$.
Thus in this case, without resorting to Lemma \ref{StEx}(d), we see $F^*(a,x)$ has bi-degree $(d,d)$
and that the tangent cone to $F^*(a,x)$ at a general point $P=a$ is defined by $F^*(x,a)$
(since here the tangent cone is defined by $F^*(a,x)$ but $F^*(a,x)$ and $F^*(x,a)$ define the same locus).
\end{example}

\begin{question} 
Is it always true for an unexpected hypersurface $F(a,x)$ of degree $d$ with a general point of multiplicity $d$ that $F^*(a,x)=\pm F^*(x,a)$?
\end{question} 

We now consider how $F^*(a,x)$ and $F^*(x,a)$ are related in the case of unexpected curves in the plane having
degree $m+1$ and a general point $P$ of multiplicity $m$. 

We begin with a lemma. Recall that $(x-x_1)(x-x_2)\cdots(x-x_n)=x^n+(-1)^1e_1x^{n-1}+(-1)^2e_2x^{n-2}+\cdots+(-1)^ne_nx^0$,
where $e_i=\sum_{1\leq j_1<j_2<\cdots<j_i\leq n}x_{j_1}\cdots x_{j_i}$ is the so-called $i$th elementary symmetric polynomial.
Let $p_i=x_1^i+\cdots+x_n^i$ be the $i$th symmetric power sum. Newton's identities (also known as the Newton-Girard formulas)
relate these as follows:
$$e_1=p_1,$$
$$2e_2=e_1p_1-p_2,$$
$$3e_3=e_2p_1-e_1p_2+p_3,$$
etc., and in general
$$ie_i=\sum_{j=1}^i(-1)^{j-1}e_{i-j}p_j.$$
Let $f:\KK^n\to \KK^n$ be the map $f(x_1,\ldots,x_n)=(e_1,\ldots,e_n)$ and let $df=(\partial e_i/\partial x_j)$ be the
matrix for the mapping on tangent spaces.

\begin{lemma}\label{diffMap} 
Given the mapping $f(x_1,\ldots,x_n)=(e_1,\ldots,e_n)$, then 
$$\det(df)(x_1,\ldots,x_n)=\pm \Pi_{i<j}(x_i-x_j)$$ 
so $\det(df)(x_1,\ldots,x_n)\neq0$ if and only if $x_i\neq x_j$ for all $i\neq j$.
\end{lemma}

\begin{proof} 
The proof is to show by applying row operations involving multiplying rows by $-1$
and adding multiples of one row to another
that one obtains the Vandermonde matrix.

Row $n$ of the matrix $df$ is the gradient of $e_n$, namely $\nabla(e_n)$. Using Newton's identity we can rewrite this as
$\frac{1}{n}\nabla(\sum_{j=1}^n(-1)^{j-1}e_{n-j}p_j)=\frac{1}{n}\sum_{j=1}^n(-1)^{j-1}(\nabla(e_{n-j})p_j+e_{n-j}\nabla(p_j))$.
Using row operations, we can (up to row equivalence) clear out the terms $\nabla(e_{n-j})p_j$ since these are multiples 
of rows $\nabla(e_i)$ higher up in the matrix. We can do the same now for row $n-1$, and then row $n-2$, etc.

Afterward, row 1 is unchanged (it is $\nabla(e_1)=\nabla(p_1)=x_1+\cdots+x_n$), but row 2 is 
$\frac{1}{2}(e_1\nabla(p_1)-\nabla(p_2))$, so we can use row 1 to clear the term $e_1\nabla(p_1)$ so that row 2 becomes
$-\nabla(p_2)$, at which point we can use rows 1 and 2 to clear row 3 so that row 3 becomes $\nabla(p_3)$.
Continuing in this way we eventually obtain a matrix row equivalent to $df$ whose rows are
$\nabla(p_1)=(1,\ldots,1)$, $-\frac{1}{2}\nabla(p_2)=-(x_1,\cdots,x_n)$, $\ldots$, $(-1)^{n-1}\frac{1}{n}\nabla(p_n)=(-1)^{n-1}(x_1^{n-1},\cdots,x_n^{n-1})$. 
Up to sign, this is the Vandermonde matrix, whose determinant is well known to be as claimed.
\end{proof}

Let $Z\subset\PP^2$ be a finite set of points admitting an unexpected curve $C$ of degree $m+1$ with a general point $P=[a_0:a_1:a_2]$
of multiplicity $m$. Let $F\in S$ be the form defining $C$ over the field $\KK(\frac{a_1}{a_0},\frac{a_2}{a_0})$. 
Examples suggest that $F^*$ is bi-homogeneous of bi-degree $(m,m+1)$, but our proof is restricted to the case that
the line arrangement dual to $Z$ is free. 
In forthcoming work, W.\ Trok \cite{Tr} establishes a more general result on the bi-degree using different methods.

\begin{theorem}\label{BMSSDualityThm}
Let $Z\subset\PP^2$ be a finite set of points admitting an irreducible unexpected curve $C=C_P$ of degree $m+1$ with a general point $P=[a_0:a_1:a_2]$
of multiplicity $m$.  
\begin{enumerate}
\item[(a)] The curve $C_P$ is unique.

\item[(b)] Let $F(a,x)\in S=\KK\big(\frac{a_1}{a_0},\frac{a_2}{a_0}\big)[x_0,x_1,x_2]$ 
be the form defining $C$ over the field $\KK(\frac{a_1}{a_0},\frac{a_2}{a_0})$.
Assume that the lines dual to the points of $Z$ comprise a free line arrangement. Then $F^*(a,x)$ is bi-homogeneous of bi-degree $(m,m+1)$.
Furthermore, viewing $F^*(a,x) \in \KK[x_0,x_1,x_2][a_0,a_1,a_2]$, $F^*(a,x)$ has multiplicity $m$ in the 
$a$ variables at the general point $[x_0:x_1:x_2]$ (briefly we will say $F^*(a,x)$ has a point of multiplicity $m$ in the $a$ variables at $a=x$).

\item[(c)] Assume that $F\in R=\KK[a_0,a_1,a_2][x_0,x_1,x_2]$ is any bi-homogeneous form of bi-degree $(m,m+1)$
such that $F(a,x)$ is reduced and irreducible for a general point $a=P$ and has multiplicity $m$ both in the $a$ variables at $a=x$ and 
in the $x$ variables at $x=a$.
Then $F_P(a,x)=(-1)^mF(x,a)$ is the tangent cone at $x=P$ to the curve $F(P,x)=0$ for $a=P$, where $F_P$ is defined in (\ref{def of HP}).
\end{enumerate}
\end{theorem}

\begin{proof}
(a) Since $C$ is irreducible, then \cite{CHMN} shows that $C$ is unique.

(b) By Lemma \ref{StEx}(d) we know that $F^*$ is bi-homogeneous of bi-degree $(m',m+1)$
for some $m'\geq m$ and that $F^*$ has multiplicity at least $m$ in the $a$ variables at $a=x$.
The conclusion follows from examining a parameterization given in \cite{CHMN} to conclude that $m'\leq m$.
Let $\Lambda$ be the product of the forms defining the lines dual to the points of $Z$.
It is not hard to check that there are no unexpected curves for $Z$ with $|Z|<3$, so
after a change of coordinates if need be, we may assume $x_0x_1$ divides $\Lambda$.
Let $(s_0,s_1,s_2)$ be a syzygy of minimal degree (hence homogeneous of degree $m$;
see \cite{CHMN}) for the ideal $(\Lambda_{x_0},\Lambda_{x_1},\Lambda_{x_2})$,
where $\Lambda_{x_i}$ denotes the partial with respect to $x_i$; thus
$$s_0\Lambda_{x_0}+s_1\Lambda_{x_1}+s_2\Lambda_{x_2}=0.$$
Define $\phi$ formally to be the vector whose components are given by the cross
product 
$$\phi=(\phi_0,\phi_1,\phi_2)=(s_0,s_1,s_2)\times (x_0,x_1,x_2).$$

Now let $\ell$ be the general line $a_0x_0+a_1x_1+a_2x_2=0$.
If $a_2\neq0$, we can parameterize $\ell$ by $(ta_2,-a_2s,a_1s-a_0t)$,
where $(t,s)$ are projective coordinates on $\PP^1$,
and by \cite{CHMN} $\phi$ defines a birational map from $\ell$ to $C$.

Note that, since $x_0x_1$ divides $\Lambda$, then $x_0$ divides $\Lambda_{x_1}$ and
$\Lambda_{x_2}$ and $x_1$ divides $\Lambda_{x_0}$ and $\Lambda_{x_2}$, but
$x_0$ does not divide $\Lambda_{x_0}$ and $x_1$ does not divide $\Lambda_{x_1}$.
Setting $x_0=0$ in $s_0\Lambda_{x_0}+s_1\Lambda_{x_1}+s_2\Lambda_{x_2}=0$ gives
$$s_0\Lambda_{x_0}=0$$
so $x_0$ must divide $s_0$. Similarly, $x_1$ divides $s_1$.

Plugging $(x_0,x_1,x_2)=(ta_2,-a_2s,a_1s-a_0t)$ into $\phi(x_0,x_1,x_2)$ we see from the definition of the cross product
that $a_2$ divides $\phi_2(ta_2,-a_2s,a_1s-a_0t)$, and since $x_i|s_i$ for $i=0,1$, we get that 
$x_i$ divides $\phi_i$ for $0,1$ and thus after the substitution $(x_0,x_1,x_2)=(ta_2,-a_2s,a_1s-a_0t)$ 
that $a_2$ divides $\phi_i$ for $i=0,1,2$. Let $\psi_i=\phi_i/a_2$. Then
$\psi_i$ is bi-homogeneous of degree $m-1$ in the $a$ variables and degree $m$ in the
variables $s$ and $t$.

We also get a parameterization of $C$ using the slope of lines through the point $P=a$. This induces an isomorphism
from $E_P$ to $C$, where $E_P$ is the blow up of $P$. But $\psi=(\psi_0,\psi_1,\psi_2)$ induces a birational map
$\ell$ to $C$. Composing gives an isomorphism $\ell$ to $E_P$, which is therefore linear.
Thus after a change of coordinates fixing $P$ and $x_2=0$, we may assume that
the parameterization of $C$ given by $\psi$ is the same as that given by slopes of lines
through $P$. In particular, consider the line through the point $(a_1/a_2,a_1/a_2,1)$ with slope $t/s$; in affine coordinates it is
$$\frac{x_1}{x_2}-\frac{a_1}{a_2}=\frac{t}{s}\Big(\frac{x_0}{x_2}-\frac{a_0}{a_2}\Big)$$
which we can rewrite as $(x_1a_2-a_1x_2)s=t(x_0a_2-a_0x_2)$. Setting $x_2=0$ gives the coordinates
$a_2t=x_1, sa_2=x_0$ as the parameterization on $E_P$, and we can assume that the parameterization
of $C$ given by mapping the point $(t,s)$ on $E_P$ to the point $(x_0:x_1:x_2)$ where the 
line $(x_1a_2-a_1x_2)s=t(x_0a_2-a_0x_2)$ meets $C$ (away from $P$) is the same point of $C$ as that given by
$\psi(ta_2,-a_2s,a_1s-a_0t)$. 

Thus the parameterization is such that 
$$[x_0:x_1:x_2]=[\psi_0(ta_2,-a_2s,a_1s-a_0t):\psi_1(ta_2,-a_2s,a_1s-a_0t):\psi_2(ta_2,-a_2s,a_1s-a_0t)],$$ 
and hence plugging $s=x_0a_2-a_0x_2$ and $t=x_1a_2-a_1x_2$ 
in to $(ta_2,-a_2s,a_1s-a_0t)$ gives $(a_2(x_1a_2-a_1x_2),-a_2(x_0a_2-a_0x_2),a_1x_0a_2-a_0x_1a_2)$, and factoring out the $a_2$ we get
$(x_1a_2-a_1x_2,-(x_0a_2-a_0x_2),a_1x_0-a_0x_1)$, each component of which has bi-degree $(1,1)$.
So we now have parametric equations for $C$,
$$[x_0:x_1:x_2]=[\psi_0(v):\psi_1(v):\psi_2(v)],$$
where $v=(x_1a_2-a_1x_2,-x_0a_2+a_0x_2,a_1x_0-a_0x_1)$, so each component of the right hand side has bi-degree $(m,m+1)$.
We get the following equations for $C$ (bi-homogeneous of degree $(m,m+2)$):
$x_0\psi_1-x_1\psi_0=0$, $x_0\psi_2-x_2\psi_0=0$, $x_1\psi_2-x_2\psi_1=0$.
The form $F^*(a,x)$ defining $C$ is a common divisor of these three equations. Thus 
$F^*$ is bi-homogeneous of bi-degree $(m',m+1)$
for some $m'\leq m$. Thus $F^*$ is bi-homogeneous of bi-degree $(m,m+1)$,
and since it has a point of multiplicity at least $m$ at $a=x$ but degree exactly $m$,
the multiplicity is also $m$.

(c) Since $F$ has degree $m$ in the $a$ variables, $F$ defines a curve
of degree $m$ with a point of multiplicity $m$ at $a=x$, hence it defines a union of $m$ lines through the
point $a=x$. 

The question remains as to what these lines are. To answer this question, 
translate the point $a=x$ to the point $[1:0:0]$. Doing this gives us
$$G(a)=a_0^mF\big(1,\frac{a_1}{a_0}+\frac{x_1}{x_0},\frac{a_2}{a_0}+\frac{x_2}{x_0},1,\frac{x_1}{x_0},\frac{x_2}{x_0}\big)
\in\KK\Big[\frac{x_1}{x_0},\frac{x_2}{x_0}\Big][a_0,a_1,a_2].$$
In the $a$ variables, this defines $m$ lines meeting at $[1:0:0]$.
Thus in the $a$ variables $G(a)$ 
is a homogeneous form of degree $m$, but now
the variable $a_0$, which we can regard as defining the line at infinity,
does not appear. 
Note that neither $a_1$ nor $a_2$ is a factor of $G(a)$, since 
if, say $a_1$ were a factor of $G$, then for all choices of $x$,
$G=0$ includes the line from $[1:0:0]$ to $[0:0:1]$. Since the point
$[0:0:1]$ is at infinity, it is fixed under the translation we employed,
so $F(a_0,a_1,a_2,x_0,x_1,x_2)=0$ includes the line from $[x_0:x_1:x_2]$ to $[0:0:1]$, hence
$x_1a_0-x_0a_1$ is a factor of $F$, contradicting the assumption that
$C$ is irreducible. 

Since $a_0$ does not appear in $G(a)$ and $a_1$ and $a_2$ are not factors,
we have $G(a)=b_0a_1^m+\cdots+b_ma_2^m$ with $b_i\in\KK\big[\frac{x_1}{x_0},\frac{x_1}{x_0}\big]$
and $b_0,b_m\neq0$. Setting $a_2=1$ in $G$ gives
$g(a_1)=b_0a_1^m+\cdots+b_m$ and
dividing by $b_0$ gives $h(a_1)=g(a_1)/b_0\in\KK(\frac{x_1}{x_0},\frac{x_1}{x_0})[a_1]$.
Over an appropriate field extension of $\KK(\frac{x_1}{x_0},\frac{x_1}{x_0})$ this factors
as $h(a_1)=(a_1-h_1)\cdots(a_1-h_m)$. 
Note that the elements $h_i$ are distinct. If not, then $h$ and hence $g$
has a multiple root, say $h_i$, hence $a_1-h_i$ is a common factor
of $g$ and $g'$ (or, taking more derivatives, of $g^{(\mu-1)}$ 
and $g^{(\mu)}$ if the root has multiplicity $\mu>1$)
and thus can be found using the Euclidean algorithm.
Thus the linear factor $a_1-h_i$ is in $\KK(\frac{x_1}{x_0},\frac{x_1}{x_0})[a_1]$ and so 
divides $g$ over $\KK(\frac{x_1}{x_0},\frac{x_1}{x_0})[a_1]$ and hence $a_1-h_ia_2$ divides $G(a)$ 
in $\KK\big[\frac{x_1}{x_0},\frac{x_1}{x_0}\big][a_0,a_1,a_2]$ 
and thus $F(a,x)$ as before has a corresponding factor linear in $a$, contrary to assumption.
Thus the roots $h_i$ are distinct.

Since by Lemma \ref{diffMap} 
the mapping $f$ expressing the coefficients of $h$ as symmetric functions of the roots $h_i$ has
an invertible differential $df$
(the roots being distinct), we can by the inverse function theorem \cite[p.\ 33]{FG} regard the roots
$h_i$ as holomorphic functions of the coefficients $b_i/b_0$ of $h$. Since each 
$b_i/b_0$ is a rational (and so holomorphic) function of $x_1/x_0$ and $x_2/x_0$, we can regard the
$h_i$ as holomorphic functions of $x_1/x_0$ and $x_2/x_0$.

The roots $h_i$ give the points at infinity where $G(a,x)$ vanishes. Specifically these are
$[0:h_i:1]$. These are unaffected by affine translations, so the points where $F(a,x)$ 
vanishes on $a_0=0$ are these same points $[0:h_i:1]$. 
Thus translating back we see $F(a,x)$ vanishes on the lines through $[x_0:x_1:x_2]$ and $[0:h_i:1]$,
so the forms defining these lines divide $F(a,x)$.
Specifically, we have $G(a)=b_0(a_1-h_1a_2)\cdots(a_1-h_ma_2)$ so
$$b_0\Big(\frac{a_1}{a_0}-h_1\frac{a_2}{a_0}\Big)\cdots\Big(\frac{a_1}{a_0}-h_m\frac{a_2}{a_0}\Big)=G(a)/a_0^m=
F\Big(1,\frac{a_1}{a_0}+\frac{x_1}{x_0},\frac{a_2}{a_0}+\frac{x_2}{x_0},1,\frac{x_1}{x_0},\frac{x_2}{x_0}\Big)$$
translates back to
$$F(a,x)=x_0^{m+1}a_0^mb_0\Big(\frac{a_1}{a_0}-\frac{x_1}{x_0}-\big(\frac{a_2}{a_0}-\frac{x_2}{x_0}\big)h_1\Big)\cdots
\Big(\frac{a_1}{a_0}-\frac{x_1}{x_0}-\big(\frac{a_2}{a_0}-\frac{x_2}{x_0}\big)h_m\Big)=$$
$$x_0b_0\big((a_1x_0-a_0x_1)-(a_2x_0-a_0x_2)h_1\big)\cdots
\big((a_1x_0-a_0x_1)-(a_2x_0-a_0x_2)h_m\big)=$$
$$b_0x_0\left|\begin{array} {ccc}
0 & h_1(x)  & 1\\
x_0 & x_1 & x_2\\
a_0 & a_1 & a_2\\
\end{array}
\right|\cdots
\left|\begin{array} {ccc}
 0 & h_m(x)  & 1\\
x_0 & x_1 & x_2\\
a_0 & a_1 & a_2\\
\end{array}
\right|.$$

Consider a point $x=a$ where $b_0x_0\neq0$. Then the branches of $F(a,x)=0$ in a neighborhood of $x=a$
are defined by the vanishing of each factor $\lambda_i(x_0,x_1,x_2)=\left|\begin{array} {ccc}
0 & h_i(x)  & 1\\
x_0 & x_1 & x_2\\
a_0 & a_1 & a_2\\
\end{array}
\right|$. The tangent line to the branch $\lambda_i(x)=0$ at $x=a$ is defined by
$((\nabla \lambda_i)|_{x=a})\cdot (x_0,x_1,x_2)$.
Since $\lambda_i=(a_1-a_2h_i)x_0-a_0x_1+a_0x_2h_i$, the gradient $\nabla \lambda_i$ is
$$\Big((a_1-a_2h_i),-a_0,a_0h_i\Big)+\Big(-a_2h_{i0}x_0+a_0x_2h_{i0},-a_2h_{i1}x_0+a_0x_2h_{i1},-a_2h_{i2}x_0+a_0x_2h_{i2}\Big).$$
The second term vanishes at $x=a$ so 
$$(\nabla \lambda_i)|_{x=a}\cdot(x_0,x_1,x_2)=(a_1-a_2h_i(x),-a_0,a_0h_i(x))|_{x=a}\cdot(x_0,x_1,x_2)=\left|\begin{array} {ccc}
0 & h_i(a)  & 1\\
x_0 & x_1 & x_2\\
a_0 & a_1 & a_2\\
\end{array}
\right|.$$

So we see that the tangent cone to $F(a,x)=0$ at a point $P=a$ such that $b_0x_0\neq0$ is defined by the vanishing of
$$F_P(a,x)=b_0(a)a_0\left|\begin{array} {ccc}
0 & h_1(a)  & 1\\
x_0 & x_1 & x_2\\
a_0 & a_1 & a_2\\
\end{array}
\right|\cdots
\left|\begin{array} {ccc}
 0 & h_m(a)  & 1\\
x_0 & x_1 & x_2\\
a_0 & a_1 & a_2\\
\end{array}
\right|=$$
$$(-1)^mb_0(a)a_0\left|\begin{array} {ccc}
0 & h_1(a)  & 1\\
a_0 & a_1 & a_2\\
x_0 & x_1 & x_2\\
\end{array}
\right|\cdots
\left|\begin{array} {ccc}
 0 & h_m(a)  & 1\\
a_0 & a_1 & a_2\\
x_0 & x_1 & x_2\\
\end{array}
\right|=(-1)^mF(x,a).$$
\end{proof}

We can now extend our result to unexpected curves that are unique but not necessarily irreducible.

\begin{corollary}\label{BMSSDualityCor}
Let $Z\subset\PP^2$ be a finite set of points admitting a unique unexpected curve $C$ of degree $m+1$ with a general point $P=[a_0:a_1:a_2]$
of multiplicity $m$. Let $F(a,x)\in S$ be the form defining $C$ over the field $\KK(\frac{a_1}{a_0},\frac{a_2}{a_0})$ and assume
$F^*$ is bi-homogeneous of bi-degree $(m,m+1)$.
Then $(F^*)_P(a,x)=(-1)^mF^*(x,a)$ defines the tangent cone to $C$ at $P$.
\end{corollary}

\begin{proof}
It is shown in \cite{CHMN} that $C=C'\cup \Lambda_1\cup\cdots\cup \Lambda_r$ where $C'$ is unexpected for a subset
$Z'$ of $Z$ and has degree $m+1-r$ with a general point $P$ of multiplicity $m-r$ where $r=|Z|-|Z'|$
and each $\Lambda_i$ is the line from $P$ to $p_i$ where $p_1,\ldots, p_r$ are the points of $Z$ not in $Z'$.
Thus $F=GL_1\cdots L_r$, where $G$ is the form defining $C'$ and $L_i$ is the bi-linear form defining $\Lambda_i$. 
So we have $F^*=G^*L_1^*\cdots L_r^*$, but $L_i^*=L_i$ and $L_i(a,x)=-L_(x,a)$, so 
$F^*(a,x)=(G^*L_1\cdots L_r)(a,x)=(-1)^m(G^*L_1\cdots L_r)(x,a)$. The tangent cone for $C$ is the tangent cone for $C'$
union with the lines $L_i$, so we also see that $(-1)^mG^*(x,a)L_1^*(x,a)\cdots L_r^*(x,a)=F^*(x,a)$ defines the tangent cone
for $C$.
\end{proof}

\begin{example}\label{QuarticUnexpCurveEx}
The form $F(a,x)$ defining the unexpected quartic curve $C$ for the $B_3$ configuration is
$$F(a,x)=a_2^3x_0^3x_1-a_2^3x_0x_1^3-a_1^3x_0^3x_2+(3a_0a_1^2-3a_0a_2^2)x_0^2x_1x_2+
(-3a_0^2a_1+3a_1a_2^2)x_0x_1^2x_2+a_0^3x_1^3x_2+$$
$$(3a_0^2a_2-3a_1^2a_2)x_0x_1x_2^2+a_1^3x_0x_2^3-a_0^3x_1x_2^3=$$
$$(x_1^3x_2-x_1x_2^3)a_0^3-3x_0x_1^2x_2a_0^2a_1+3x_0^2x_1x_2a_0a_1^2+(-x_0^3x_2+x_0x_2^3)a_1^3+3x_0x_1x_2^2a_0^2a_2$$
$$-3x_0x_1x_2^2a_1^2a_2-3x_0^2x_1x_2a_0a_2^2+3x_0x_1^2x_2a_1a_2^2+(x_0^3x_1-x_0x_1^3)a_2^3.$$
Substituting $a_0=x_0=1$ and restricting to the line $L$ defined by $x_1=a_1$ gives
$$F(1,a_1,a_2,1,a_1,x_2)=a_1(a_1^2-1)(x_2-a_2)^3,$$
while substituting $a_0=x_0=1,a_1=x_1$ gives
$$F(1,x_1,a_2,1,x_1,x_2)=x_1(x_1^2-1)(x_2-a_2)^3.$$
Thus we see that $F$ has a triple point both in terms of the $a$ variables and the $x$ variables.
To see that the factors of $F(x,a)$
are just the lines tangent to the branches of $F(a,x)$,
one can graph $F(a,x)$ and $F(x,a)$ on the same coordinate axes; see Figure \ref{TCFig}.
\end{example}

\begin{figure}
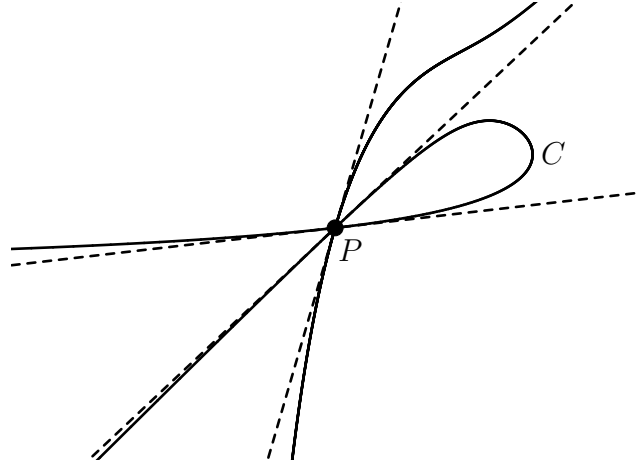

\begin{center}

\caption{The unexpected $B_3$ quartic curve $C$ defined by $F(a,x)=0$ (graphed as a solid line) and 
the graph of $F(x,a)=0$ (dashed lines)
for $P=(a_0,a_1,a_2)=(-6,-5,4)$.}
\label{TCFig}
\end{center}
\end{figure}

\begin{remark}\label{unexpNotFree} The line arrangement of lines dual to the 9 points of $Z$ giving the irreducible unexpected curve in 
Example \ref{QuarticUnexpCurveEx} (i.e., the $B_3$ arrangement of 9 lines) is free.
This raises the question of whether every irreducible unexpected curve comes from a free arrangement.
It is not true that irreducible unexpected curves never come from non-free line arrangements.
For example, consider Figure \ref{20lines} (which comes from \cite[Example 6.2]{CHMN}); 
it gives a non-free arrangement of 19 lines (these 19 being the 18 solid lines shown in the figure
plus the line at infinity). The 19 points dual to these 19 lines gives a set of points $Z$ having an irreducible unexpected
curve of degree 9 with a general point of multiplicity 8. However, by including the dotted line shown in the figure, we obtain 
a free arrangement of 20 lines for which the dual set $Z'$ of 20 points has the same unexpected curve.
\end{remark}

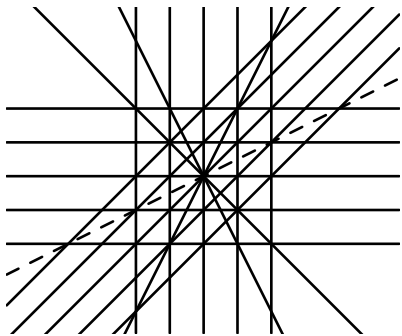
\begin{figure}
\begin{center}
\begin{tikzpicture}[line cap=round,line join=round,>=triangle 45,x=.45cm,y=.45cm]
\clip(-5.82,-4.64) rectangle (5.78,5);
\draw [line width=1.pt] (0.,-4.64) -- (0.,5.24);
\draw [line width=1.pt,domain=-5.82:5.78] plot(\x,{(-0.-0.*\x)/1.});
\draw [line width=1.pt,domain=-5.82:5.78] plot(\x,{(-0.-1.*\x)/1.});
\draw [line width=1.pt,domain=-5.82:5.78] plot(\x,{(-0.-1.*\x)/-1.});
\draw [line width=1.pt,domain=-5.82:5.78] plot(\x,{(-0.-2.*\x)/1.});
\draw [line width=1.pt,domain=-5.82:5.78] plot(\x,{(-0.-2.*\x)/-1.});
\draw [line width=1.pt] (-1.,-4.64) -- (-1.,5.24);
\draw [line width=1.pt] (1.,-4.64) -- (1.,5.24);
\draw [line width=1.pt,domain=-5.82:5.78] plot(\x,{(-1.-0.*\x)/1.});
\draw [line width=1.pt,domain=-5.82:5.78] plot(\x,{(--1.-0.*\x)/1.});
\draw [line width=1.pt] (-2.,-4.64) -- (-2.,5.24);
\draw [line width=1.pt] (2.,-4.64) -- (2.,5.24);
\draw [line width=1.pt,domain=-5.82:5.78] plot(\x,{(-2.-0.*\x)/1.});
\draw [line width=1.pt,domain=-5.82:5.78] plot(\x,{(--2.-0.*\x)/1.});
\draw [line width=1.pt,domain=-5.82:5.78] plot(\x,{(-1.-1.*\x)/-1.});
\draw [line width=1.pt,domain=-5.82:5.78] plot(\x,{(--1.-1.*\x)/-1.});
\draw [line width=1.pt,domain=-5.82:5.78] plot(\x,{(-2.-1.*\x)/-1.});
\draw [line width=1.pt,domain=-5.82:5.78] plot(\x,{(--2.-1.*\x)/-1.});
\draw [line width=1.pt,dash pattern=on 5pt off 5pt,domain=-5.82:5.78] plot(\x,{(-0.-1.*\x)/-2.});
\end{tikzpicture}
\caption{A non-free arrangement of 19 lines (these being the 18 solid lines plus the line $z=0$ at infinity, which is not shown)
which becomes free after adding a line (the dotted line) and such that both arrangements have the same
irreducible unexpected curve.}
\label{20lines}
\end{center}
\end{figure}


\section{Open Problems} \label{open problems}

In this short section we list some open problems stemming from this work.

\begin{enumerate}

\item Suppose $Z$ is a finite set of points which admits an unexpected curve of degree $d=m+1$ having a general point 
$P$ of multiplicity $m$. If the arrangement of lines dual to $Z$ is not free, to what extent does BMSS duality still hold? 

\item  To what extent does BMSS duality hold in higher dimensions? 
For example, it holds for $B_4$ and $F_4$ with $d=m=4$. In these cases the unexpected surfaces are defined by a form 
$F(a,x)$ of bi-degree $(4,4)$ and the form for the tangent cone at $P$ is $F(x,a)$.
It also holds for $D_4$ with $d=m=3$. In this case the unexpected surfaces are defined by a form $F(a,x)$ of bi-degree 
$(3,3)$ and the form for the tangent cone at $P$ is $-F(x,a)$. 

\item  Given an unexpected variety for a finite point set $Z$ having a general point $P$ of multiplicity $m$ and degree $d$, 
let $B_Z(P)$ be the base locus of $[I_{Z+mP}]_d$ and let $B_Z=\cap_P B_Z(P)$, which we can refer to as the base locus associated 
to $Z$. What can be said about this associated base locus? For example, what is its dimension? If it is 0-dimensional, when is it strictly larger than $Z$?

\item  What is special about the root systems having unexpected hypersurfaces? For example, why do the systems $A_{n+1}$ not seem to have any?

\item  For the systems $B_{n+1}$, computational runs suggest there might be unexpected hypersurfaces with $d=m=4$ for all $n\geq 3$ and for $d=m=3$ for all $n\geq5$. How can one prove this? And why only $3\leq m\leq 4$?

\item Let $Z$ be a non-degenerate set of points in linear general position in $\PP^n$, $n \geq 3$. Is it true that there does not exist an unexpected hypersurface of any degree $d$ and multiplicity $d-1$ at a general point?

\item Is there a class of finite sets of points in $\PP^n$ for $n \geq 3$ (or respectively a condition on $(d,m)$) for which the syzygy 
bundle plays a similar role, in the study of unexpected hypersurfaces, to that which it plays when $n=2$ and $m = d-1$, or is that purely a phenomenon for the plane?

\item Let $Z$ be a set of points in $\PP^2$ admitting a (unique) unexpected curve of degree $m+1$ with a general point of multiplicity $m$. 
Then $Z + mP$ does not impose independent conditions on forms of degree $m+1$, but for a suitable subset $Z'$ of $2m+2$ of the points of $Z$, $Z' + mP$ 
does impose independent conditions, and the curve is still unique (but it is no longer unexpected for $Z'$). 

So suppose we consider more generally sets $Z$ of $2m+2$ points such that there is a unique irreducible curve of degree $m+1$ containing $Z$ and having a 
general point of multiplicity $m$ (but not necessarily unexpected). How does the bi-degree of $F^*$ depend on $Z$? Is there a connection between this bi-degree 
and the question of whether $Z$ extends to a set of points for which the curve is unexpected? 

\item Suppose $Z\subset\PP^2$ is a finite set of points having an irreducible unexpected curve where (as in Remark \ref{unexpNotFree})
the arrangement of lines dual to $Z$ is not free. Is it true that there is a finite set $Z'$ with $Z\subset Z'\subset\PP^2$ such that $Z'$ has 
the same unexpected curve as does $Z$ but such that the arrangement of lines dual to $Z'$ is free?

\end{enumerate}



\begin{thebibliography}{BDHHSS}

\bibitem[A]{A}
S.\ Akesseh, {\em Ideal containments under flat extensions and interpolation on linear systems in $\PP^2$}, 
PhD Thesis, The University of Nebraska-Lincoln, 2017.

\bibitem[AT]{AT}
B.\ Anzis and \c{S}.\ Toh\v{a}neanu, {\em On the geometry of real or complex supersolvable line arrangements}, 
Journal of Combinatorial Theory, Series A, 140 : 76--96 (2016), arXiv:1501.04039v2.

\bibitem[BFGMV]{BFGMV}
C.\ Bao, C.\ Freidman-Gerlicz, G.\ Gordon, P.\ McGrath and J.\ Vega,
{\em Matroid automorphisms of the $H_4$ root system}, Congr. Numer. 207 (2011) 141--160.
(arXiv:1005.5492v2).

\bibitem[BMSS]{BMSS}
T.\ Bauer, G.\ Malara, T.\ Szemberg and J.\ Szpond, 
{\em Quartic unexpected curves and surfaces.}, 
Manuscripta Math. (2018) \url{https://doi.org/10.1007/s00229-018-1091-3} (arXiv:1804.03610). 

\bibitem[BV]{BV}
W.\ Bruns and U.\ Vetter,
{\em Determinantal Rings},
Lecture Notes in Mathematics 1327, 1988.

\bibitem[CoCoA]{cocoa}   CoCoATeam,
  CoCoA: a system for doing
     Computations in Commutative Algebra,
  Available at http://cocoa.dima.unige.it

\bibitem[CHMN]{CHMN}
D.\ Cook, B.\ Harbourne, J.\ Migliore and U.\ Nagel, 
{\em Line arrangements and configurations of points with an unexpected geometric property}, 
Compositio Math. 154:10 (2018) 2150--2194 (arXiv:1602.02300). 

\bibitem[DI]{DI} 
R.\ Di Gennaro and G.\ Ilardi, 
{\em Laplace equations, Lefschetz properties and line arrangements}, 
J.\ Pure Appl.\ Algebra 222 (2018), no.\ 9, 2657--2666.

\bibitem[DIV]{DIV} 
R.\ Di Gennaro, G.\ Ilardi and J.\ Vall\`es,  
{\em Singular hypersurfaces characterizing the Lefschetz properties}, 
J.\ London Math.\ Soc. (2) {\bf 89} (2014), no.\ 1, 194--212 (arXiv:1210.2292).

\bibitem[DMO]{DMO}
M.\ Di Marca, G.\ Malara, A.\ Oneto, {\em Unexpected curves arising from special line arrangements}, preprint, arXiv:1804.02730.

\bibitem[EI]{EI}
J.\ Emsalem and A.\ Iarrobino, 
{\em Inverse system of a symbolic power I}, J.\ Algebra 174 (1995) 1080--1090.

\bibitem[FV]{FV}
D.\ Faenzi and J.\ Vall\`es, {\em Logarithmic bundles and line arrangements, an approach via the standard construction},
J. Lond. Math. Soc. (2) 90 (2014), no. 3, 675--694 (arXiv:1209.4934)

\bibitem[FGST]{FGST} 
\L.\ Farnik, F.\ Galuppi, L.\ Sodomaco and W.\ Trok,
{\em On the unique unexpected quartic in $\PP^2$}, preprint arXiv:1804.03590, 2018.

\bibitem[FG]{FG}
K.\ Fritzsche, H.\ Grauert, 
{\em From Holomorphic Functions to Complex Manifolds}, Springer-Verlag, (2002). 


\bibitem[GLP]{GLP}
L.\ Gruson, R.\ Lazarsfeld and C.\ Peskine, {\em On a theorem of Castelnuovo, and the equations defining space curves}, 
Invent.\ Math.\ 72 (1983), no. 3, 491--506. 


\bibitem[Har]{harris}
J. Harris, {\em The genus of space curves}, Math. Ann. {\bf 249} (1980), 191--204.


\bibitem[HSS]{HSS}
B.\  Harbourne, H.\ Schenck and A.\ Seceleanu, {\em  Inverse systems, Gelfand-Tsetlin patterns 
and the weak Lefschetz property.}, J.\ Lond. Math. Soc. (2) 84 (2011), no. 3, 712--730.

\bibitem[I]{I}
G.\ Ilardi, {\em Jacobian ideals, arrangements and the Lefschetz properties}, 
J.\ Algebra 508 (2018) 418--430. 


\bibitem[M2]{M2}
D.\ R.\ Grayson and M.\ E.\ Stillman, 
{\em Macaulay2, a software system for research in algebraic geometry},
Available at \url{http://www.math.uiuc.edu/Macaulay2/}. 

\bibitem[Me]{Me}
E.\ Melchior, {\em \"Uber Vielseite der projektiven Ebene}, Deutsche Math., 5 (1940),
461--475.

\bibitem[Mi]{migbook}
J. Migliore, ``Introduction to liaison theory and deficiency modules," Birkh\"auser, 
Progress in Mathematics vol. {\bf 165} (1998).

\bibitem[MMN]{MMN}
J.\ Migliore, R.\ Mir\'o-Roig and U.\ Nagel, {\em On the weak Lefschetz property for 
powers of linear forms}, Algebra Number Theory 6 (2012), no. 3, 487--526.

\bibitem[SS]{SS}
H. Schenck and A. Seceleanu, {\em The weak Lefschetz property and powers of linear
forms in $K[x, y, z]$}, Proc. Amer. Math. Soc. 138 (2010), no. 7, 2335--2339.


\bibitem[St]{St}
J.\ Stembridge,
{\em A construction of $H_4$ without miracles}, Discrete Comput. Geom. 22 (1999), 425--427.

\bibitem[Tr]{Tr} W. Trok, {\em Unexpected Hypersurfaces and Fat Linear Subspaces of Codimension 2}, in preparation (2019).

\end{thebibliography}
\end{document}